\documentclass[a4paper,12pt]{amsart}
\usepackage{amsmath}
\usepackage{amssymb}
\usepackage{amscd}
\usepackage{mathrsfs}
\usepackage[dvipdfmx]{graphicx}
\everymath{\displaystyle}
\usepackage[usenames]{color}

\theoremstyle{plain} 
\newtheorem{theorem}{\indent\sc Theorem}[section]
\newtheorem{lemma}[theorem]{\indent\sc Lemma}
\newtheorem{corollary}[theorem]{\indent\sc Corollary}
\newtheorem{proposition}[theorem]{\indent\sc Proposition}
\newtheorem{claim}[theorem]{\indent\sc Claim}

\theoremstyle{definition} 
\newtheorem{definition}[theorem]{\indent\sc Definition}
\newtheorem{remark}[theorem]{\indent\sc Remark}

\newtheorem{question}[theorem]{\indent\sc Question}

\begin{document}
\pagestyle{plain}
\thispagestyle{plain}

\title[Higher codimensional Ueda theory]
{Higher codimensional Ueda theory for a compact submanifold with unitary flat normal bundle}

\author[Takayuki KOIKE]{Takayuki KOIKE}
\address{ 
Department of Mathematics, Graduate School of Science, Kyoto University, Kyoto 606-8502, Japan. 
}
\email{tkoike@math.kyoto-u.ac.jp}


\begin{abstract}
Let $Y$ be a compact complex manifold embedded in a complex manifold with unitary flat normal bundle. 
Our interest is in a sort of the linearizability problem of a neighborhood of $Y$. 
As a higher-codimensional generalization of Ueda's result, 
we give a sufficient condition for the existence of a non-singular holomorphic foliation on a neighborhood of $Y$ which includes $Y$ as a leaf with unitary-linear holonomy. 
We apply this result to the existence problem of a smooth Hermitian metric with semi-positive curvature on a nef line bundle. 
\end{abstract}

\maketitle

\section{Introduction}

Let $X$ be a complex manifold and $Y$ be a complex submanifold of codimension $r$. 
Our interest is in the analytic structure of a neighborhood of $Y$ when the normal bundle $N_{Y/X}$ is unitary flat. 
We say that a holomorphic vector bundle $E$ on $Y$ is {\it unitary flat} if $E\in {\rm Image}\,(H^1(Y, U(r))\to H^1(Y, GL_r(\mathcal{O}_Y)))$, or equivalently, the transition matrices of $E$ can be chosen to be $U(r)$-valued locally constant functions, where $U(r)$ is the set of $r\times r$ unitary matrices. 
A unitary flat vector bundle $E$ admits a flat connection whose monodromy $\rho_{E}$ is a unitary representation of the fundamental group $\pi_1(Y, *)$ of $Y$ (see \S \ref{section:unitary_flat} for the details). 
Our interest is in a sort of the {\it linearizability problem} of a neighborhood of $Y$. 
In other words, we are interested in comparing a neighborhood of $Y$ in $X$ and of the zero section in $N_{Y/X}$.   
One main goal of this paper is to investigate the existence of a holomorphic foliation $\mathcal{F}$ of codimension $r$ on a neighborhood of $Y$ which includes $Y$ as a leaf with ${\rm Hol}_{\mathcal{F}, Y}=\rho_{N_{Y/X}}$, 
where ${\rm Hol}_{\mathcal{F}, Y}$ is the holonomy of $\mathcal{F}$ along $Y$. 

In \cite{A}, Arnol'd studied a neighborhood of an elliptic curve $Y$ by applying a linearizing technique as in \cite{S}. 
In \cite{U}, Ueda studied the case where $Y$ is any compact complex curve and $r=1$. 
For such a pair $(Y, X)$, Ueda posed so-called the {\it Ueda class} $u_n(Y, X)\in H^1(Y, N_{Y/X}^{-n}):=H^1(Y, \mathcal{O}_Y(N_{Y/X}^{\otimes-n}))$ as an obstruction of the linearization of a neighborhood of $Y$ in $n$-jet along $Y$ ($n\geq 1$). 
The pair $(Y, X)$ is of {\it infinite type} if $u_n(Y, X)=0$ holds for each $n\geq 1$. 
When $(Y, X)$ is of infinite type, Ueda generalized the result of Arnol'd. 
He showed that an infinite type pair $(Y, X)$ admits the foliation $\mathcal{F}$ as above if $N_{Y/X}$ is a torsion element of $\mathcal{P}(Y):={\rm Image}\,(H^1(Y, U(1))\to H^1(Y, \mathcal{O}_Y^*))$, or satisfies the following Diophantine-type condition: 
there exists a constant $A>0$ such that $d(\mathbb{I}^{(1)}_Y, N_{Y/X}^m)\geq (2m)^{-A}$ holds for any $m\geq 1$. 
Here $\mathbb{I}^{(1)}_Y$ is the holomorphically trivial line bundle on $Y$ and 
$d$ is an invariant distance on $\mathcal{P}(Y)$ (\cite[Theorem 3]{U}, see \cite[\S 4.1]{U} for the details). 
Note that the proof of \cite[Theorem 3]{U} works not only when $Y$ is a curve, but also when $Y$ is a compact complex manifold of arbitrary dimension. 

We generalize Ueda's theory to the case where the codimension $r$ of $Y$ is greater than $1$. 
We will define the obstruction class $u_n(Y, X)\in H^1(Y, N_{Y/X}\otimes S^{n+1}N_{Y/X}^*)$ as a straightforward generalization of the Ueda class and generalize \cite[Theorem 3]{U} to the pair $(Y, X)$ of infinite type with $N_{Y/X}\in \mathcal{E}_0^{(r)}(Y)\cup \mathcal{E}_1^{(r)}(Y)$. 
Here we denote by $S^{n+1}N_{Y/X}^*$ the $n+1$-th symmetric tensor bundle, $\mathcal{E}_0^{(r)}(Y)$ the set $\{E_\rho \mid \#({\rm Image}\,\rho)<\infty\}$, 
and by $\mathcal{E}_1^{(r)}(Y)$ the set 
\[
\left\{E_\rho\left| \pi^*E_\rho\in \mathcal{S}^{(r)}_A(\widetilde{Y})\ {\rm for}\ {\rm some}\ {\rm finite}\ {\rm normal}\ {\rm covering}\ \pi\colon\widetilde{Y}\to Y\ {\rm and}\ A>0\right.\right\}, 
\]
where $E_\rho$ is the unitary flat vector bundle of rank $r$ which corresponds to a unitary representation $\rho$ of $\pi_1(Y, *)$ (see \S \ref{section:unitary_flat} for the correspondence) and 
$\mathcal{S}^{(r)}_A(\widetilde{Y})$ is the set 
\[
\left\{\bigoplus_{\lambda=1}^rL_\lambda\left| L_\lambda\in \mathcal{P}(\widetilde{Y}),\ d\left(\mathbb{I}_{\widetilde{Y}}^{(1)},\ \textstyle\bigotimes_{\lambda=1}^r L_\lambda^{a_\lambda}\right)\geq \frac{1}{\left(2|a|\right)^A}\ {\rm for}\ a=(a_\lambda)_\lambda\in\mathbb{Z}^r\ {\rm with}\ |a|\geq 1 \right\}\right.
\]
($|a|:=a_1+a_2+\cdots+a_r$). 

\begin{theorem}\label{thm:main}
Let $X$ be a complex manifold and $Y$ be a complex submanifold of codimension $r$ with $N_{Y/X}\in\mathcal{E}_0^{(r)}(Y)\cup\mathcal{E}_1^{(r)}(Y)$. 
Assume that the pair $(Y, X)$ is of infinite type (i.e. $u_n(Y, X)=0$ for each $n\geq 1$). 
Then the following holds: \\
$(i)$ There exists a non-singular holomorphic foliation $\mathcal{F}$ of codimension $r$ on some neighborhood $V$ of $Y$ which includes $Y$ as a leaf with ${\rm Hol}_{\mathcal{F}, Y}=\rho_{N_{Y/X}}$. \\
$(ii)$ For each hypersurface $S$ such that $Y\subset S$ and $N_{Y/S}$ is unitary flat, there exists a non-singular holomorphic foliation $\mathcal{G}_S$ of codimension $1$ on $V$ with the following properties by shrinking $V$ if necessary: 
$\mathcal{G}_S$ includes $S\cap V$ as a leaf with $U(1)$-linear holonomy, 
and each leaf of $\mathcal{F}$ is holomorphically immersed into a leaf of $\mathcal{G}_S$. 
Especially, $S\cap V$ is the union of leaves of $\mathcal{F}$. 
\end{theorem}

Note that the assertion $(i)$ in Theorem \ref{thm:main} implies the existence of a deformation family of $Y$ when $N_{Y/X}\in\mathcal{E}_0^{(r)}(Y)$. 
Note also that, for each unitary flat subbundle $E\subset N_{Y/X}$, the assertion $(i)$ implies the existence of a submanifold $Z\subset V$ with $Y\subset Z$ and $N_{Y/Z}=E$ (see Remark \ref{rmk:unitary_flat_subbundle}). 
Especially, when $N_{Y/X}$ admits a direct decomposition to $r$ unitary flat line bundles, it is observed from Theorem \ref{thm:main} $(i)$ that $Y$ is realized as a transversal intersection of $r$ non-singular hypersurfaces of a neighborhood of $Y$ in $X$ if $N_{Y/X}\in\mathcal{E}_0^{(r)}(Y)\cup\mathcal{E}_1^{(r)}(Y)$ and the pair $(Y, X)$ is of infinite type. 
As the pair $(Y, X)$ is always of infinite type when $H^1(Y, N_{Y/X}\otimes S^{n+1}N_{Y/X}^*)=0$ for each $n\geq 1$, Theorem \ref{thm:main} $(i)$ can be applied to the case where, for example, $Y$ is an elliptic curve and $N_{Y/X}\in\mathcal{E}^{(r)}_1(Y)$. 
In this sense, Theorem \ref{thm:main} $(i)$ can be also regarded as a generalization of the result of Arnol'd for elliptic curves. 

Theorem \ref{thm:main} $(ii)$ can be applied to the semi-positivity problem (the existence problem of a $C^\infty$ Hermitian metric with semi-positive curvature) on a holomorphic line bundle, 
since the assertion $(ii)$ implies the unitary flatness of the line bundle $[S]$ on $V$. 

\begin{corollary}\label{cor:main}
Let $X$ be a complex manifold of dimension $n$ and $L$ be a holomorphic line bundle on $X$. 
Take $D_1, D_2, \dots, D_{n-1}\in |L|$. 
Assume that $C:=\textstyle\bigcap_{\lambda=1}^{n-1} D_\lambda$ is a smooth elliptic curve, $L|_C\in \mathcal{E}_1^{(1)}(C)$, and $\{D_\lambda\}_{\lambda=1}^{n-1}$ intersects transversally along $C$. 
Then $L$ is semi-positive (i.e. $L$ admits a $C^\infty$ Hermitian metric with semi-positive curvature). 
\end{corollary}

Note that $L$ as in Corollary \ref{cor:main} has $C$ as a stable base locus: $C=SB(L):=\textstyle\bigcap_{m\geq 1}{\rm Bs}\,|L^m|$. 
Corollary \ref{cor:main} can be applied to the example of the blow-up of a del Pezzo manifold at a general point as follows: 

\begin{corollary}\label{cor:dPmfd_metric}
Let $(V, L)$ be a del Pezzo manifold of degree $1$ 
(i.e. $V$ is a projective manifold of dimension $n$ and $L$ is an ample line bundle on $V$ with $K_V^{-1}\cong L^{n-1}$ and the self-intersection number $(L^n)$ is equal to $1$), 
and $C\subset V$ be an intersection of general $n-1$ elements of $|L|$. 
For each point $q\in C$ with $L|_C\otimes [-q]\in \mathcal{E}_1^{(0)}(C)\cup\mathcal{E}_1^{(1)}(C)$, the anti-canonical bundle of the blow-up of $V$ at $q$ is semi-positive. 
\end{corollary}

We remark that 
Corollary \ref{cor:dPmfd_metric} can be regarded as a generalization of the known phenomena for the blow-up of $\mathbb{P}^2$ at general nine points (\cite{A}, \cite{B}, \cite{U}, see also \cite[\S 1]{D}), or 
the blow-up of $\mathbb{P}^3$ at general eight points 
(\cite[Corollary 1]{K}. Note that this result can be re-proved by using \cite[Theorem 1.4, Remark 3.12]{KO}, which is a corrected form of \cite[Theorem 1]{K}). 

The organization of the paper is as follows. 
In \S 2, we summarize some fundamental facts and notations on the unitary flat vector bundles on a compact complex manifold and local defining functions of compact submanifolds. 
In \S 3, we give the definitions of the obstruction class $u_n(Y, X)$ and the type of the pair $(Y, X)$. 
In \S 4, we prove Theorem \ref{thm:main}. 
In \S 5, we show Corollary \ref{cor:main}. 
In \S 6, we give some examples. Here we will prove Corollary \ref{cor:dPmfd_metric}. 
In \S 7, we list some remaining problems. 

\vskip3mm
{\bf Acknowledgment. } 
The author would like to give heartful thanks to Prof. Tetsuo Ueda whose comments and suggestions were of inestimable value for my study. 
He thanks Prof. Kento Fujita who taught him an example of the blow-up of a del Pezzo manifold at a general point.  
He also thanks Prof. Noboru Ogawa for helpful comments and warm encouragements. 
He is supported by the Grant-in-Aid for Scientific Research (KAKENHI No.28-4196) and the Grant-in-Aid for JSPS fellows. 


\section{Preliminaries}

\subsection{Unitary flat vector bundles on compact complex manifolds}\label{section:unitary_flat}

Let $Y$ be a compact complex manifold 
and $E$ be a holomorphic vector bundle on $Y$. 
We say that $E$ is {\it unitary flat} if $E\in {\rm Image}\,(H^1(Y, U(r))\to H^1(Y, GL_r(\mathcal{O}_Y)))$. 
It means that, for a suitable choice of an open covering $\{U_j\}$ of $Y$ and a local frame $e_j=(e_j^1, e_j^2, \dots, e_j^r)$ of $E$ on each $U_j$, the transition matrix $T_{jk}$ of $\{(U_j, e_j)\}$ on each $U_{jk}:=U_j\cap U_k$ can be a locally constant function values in $U(r)$: 
i.e. for some $T_{jk}\in U(r)$, it holds that $e_j=T_{jk}e_k$, or equivalently, $e_j^\lambda=\textstyle\sum_{\mu=1}^r(T_{jk})^\lambda_\mu\cdot e_k^\mu$. 
Here we denote by $(T_{jk})^\lambda_\mu$ the $(\lambda, \mu)$-th  entry of $T_{jk}$. 
For a unitary flat vector bundle $E$, we can define a {\it unitary flat metric} $h$ on $E$ by regarding each $e_j$ as an orthonormal frame. 
By using this $h$, we obtain: 

\begin{lemma}\label{lem:unitary_section_const}
Let $a_{j, \lambda}\colon U_j\to \mathbb{C}$ be a holomorphic function. 
Assume that $\{(U_j, \textstyle\sum_{\lambda=1}^ra_{j, \lambda}\cdot e_j^\lambda)\}$ glue up to define a holomorphic global section $a$ of $E$. 
Then $a_{j, \lambda}$ is a locally constant function on each $U_j$. 
\end{lemma}

\begin{proof}(see also the proof of {\cite[\S 1 Proposition 1]{Se}})
By applying the maximal principle to the psh (plurisubharmonic) function $|a|_h^2$, 
we obtain $|a|_h^2\equiv C$ for some constant $C$. 
As it holds that $|a_{j, \lambda}|^2=C-\textstyle\sum_{\lambda\not=\mu}|a_{j, \mu}|^2$ on each $U_j$, we conclude that $|a_{j, \lambda}|^2$ is pluriharmonic for each $\lambda=1, 2, \dots, r$, which proves the lemma. 
\end{proof}

By considering the monodromy of the Chern connection of $h$, we obtain a unitary representation $\rho=\rho_E\colon \pi_1(Y, *)\to U(r)$. 
Conversely, for a given unitary representation $\rho\colon\pi_1(Y, *)\to U(r)$, we can construct a unitary flat vector bundle $E_\rho$ by 
\[
E_\rho:=\widetilde{Y}\times \mathbb{C}^r/\sim_\rho, 
\]
where $\widetilde{Y}\to Y$ is the universal covering of $Y$ and $\sim_\rho$ is the relation defined by $(z, v)\sim_\rho (\gamma z, \rho(\gamma) v)$ for each $(z, v)\in \widetilde{Y}\times \mathbb{C}^r$ and $\gamma\in\pi_1(Y, *)$. 

\begin{proposition}\label{prop:flat_rep_corresp}
The above gives $1:1$-correspondence between the image of the natural map
$H^1(Y, U(r))\to H^1(Y, GL_r(\mathcal{O}_Y))$ and the set 
$\{\rho: U(r){\rm -representation}\ {\rm of}\ \linebreak\pi_1(Y, *)\}/\sim$, 
where $\rho\sim\rho'$ means that there exists $A\in GL_r(\mathbb{C})$ such that $A^{-1}\cdot \rho\cdot A=\rho'$ holds. 
\end{proposition}

For proving Proposition \ref{prop:flat_rep_corresp}, we need the following: 

\begin{lemma}\label{lem:gl_flat_str_unique}
Let $E$ and $F$ be unitary flat vector bundles on $Y$. 
Assume that $E$ and $F$ are isomorphic to each other as holomorphic vector bundles. 
Then the image of $E$ and $F$ by the natural map $H^1(Y, U(r))\to H^1(Y, GL_r(\mathbb{C}))$ coincide with each other. 
\end{lemma}

\begin{proof}
The lemma is shown by applying Lemma \ref{lem:unitary_section_const} to a global section of the unitary flat vector bundle ${\rm Hom}(E, F)\cong E^*\otimes F$. 
See the proof of {\cite[\S 1 Proposition 1]{Se}} for the details. 
\end{proof}

\begin{proof}[Proof of Proposition \ref{prop:flat_rep_corresp}]
Let $E$ be a unitary flat vector bundle. 
Take $\{(U_j, e_j)\}$ and $\rho=\rho_E$ as above. 
Let $\{(U_j, f_j=(f_j^1, f_j^2, \dots, f_j^r))\}$ be another local frame of $E$ with $f_j=S_{jk} f_k$ on each $U_{jk}$ ($S_{jk}\in U(r)$), 
and $\rho'\colon \pi_1(Y, *)\to U(r)$ be the monodromy defined by using $f_j$ as an orthonormal frame on each $U_j$. 
By Lemma \ref{lem:gl_flat_str_unique}, we can take $A_j\in GL_r(\mathbb{C})$ for each $j$ with $A_jS_{jk}=T_{jk}A_k$. 
Then, for each loop $\gamma$ of $Y$ with a base point $*\in U_j$, we can calculate that $\rho([\gamma])=A_j\cdot \rho'([\gamma])\cdot A_j^{-1}$, which proves $\rho\sim\rho'$. 

Conversely, let $\rho$ and $\rho'$ be two $U(r)$-representations of $\pi_1(Y, *)$ with $\rho\sim \rho'$. 
Take $A\in GL_r(\mathbb{C})$ such that $A^{-1}\cdot \rho\cdot A=\rho'$. 
Define the map $F\colon \widetilde{Y}\times \mathbb{C}^r\to \widetilde{Y}\times \mathbb{C}^r$ by $F(z, v):=(z, A^{-1}\cdot v)$. 
Then it is easily observed that $F$ induces an isomorphism $E_\rho\cong E_{\rho'}$, which proves the proposition. 
\end{proof}

\begin{remark}\label{rmk:injectivity_of_the_natural_map}
The definition of the relation $\sim$ in Proposition \ref{prop:flat_rep_corresp} can be replaced by the following one: 
we say $\rho\sim\rho'$ if there exists $U\in U(r)$ such that $U^{-1}\cdot \rho\cdot U=\rho'$ holds. 
It is because, for each $A\in GL_r(\mathbb{C})$ and $S\in U(r)$ with $A^{-1}\cdot S\cdot A\in U(r)$, it holds that $A^{-1}\cdot S\cdot A=U_A^{-1}\cdot S\cdot U_A$, where 
$U_A$ is the unitary part of the polar decomposition $A=U_A\cdot P_A$. 
Therefore, by the same argument as in the proof of Proposition \ref{prop:flat_rep_corresp}, we obtain that the image of $H^1(Y, U(r))\to H^1(Y, GL_r(\mathcal{O}_Y))$ is naturally isomorphic to $H^1(Y, U(r))$, 
or equivalently, 
the natural map $H^1(Y, U(r))\to H^1(Y, GL_r(\mathcal{O}_Y))$ is injective. 
\end{remark}

\begin{remark}
Here we give another (more direct) proof of the injectivity of the natural map $i\colon H^1(Y, U(r))\to H^1(Y, GL_r(\mathcal{O}_Y))$, which was taught by Professor Tetsuo Ueda. 
Let $E:=\{(U_{jk}, T_{jk})\}$ and $F:=\{(U_{jk}, S_{jk})\}$ be elements of $H^1(Y, U(r))$ with $i(E)=i(F)$. 
By Lemma \ref{lem:gl_flat_str_unique}, we can take $A_j\in GL_r(\mathbb{C})$ for each $j$ such that $S_{jk}A_k=A_jT_{jk}$ holds. 
Denote by $A_j=P_jU_j$ the polar decomposition of $A_j$, where $P_j$ is the positive definite Hermitian part and $U_j$ is the unitary part. 
Then we have $(S_{jk}P_kS_{jk}^{-1})\cdot (S_{jk}U_k)=P_j\cdot (U_jT_{jk})$. 
By the uniqueness of the polar decomposition, we obtain $S_{jk}U_k=U_jT_{jk}$. 
\end{remark}

\begin{remark}\label{rmk:unitary_flat_subbundle}
Let $F$ be a holomorpchic subbundle of a unitary flat vector bundle $E$ on $Y$. 
Then it follows from the same argument as in the proof of Lemma \ref{lem:gl_flat_str_unique} that $F$ is a unitary flat subbundle of $E$: 
i.e. $F$ is the unitary flat vector bundle which corresponds to a unitray subrepresentation of $\rho_E$. 
\end{remark}

\subsection{Local defining functions}\label{section:notation}

Let $X$ be a complex manifold $X$ and $Y$ be a compact complex submanifold of codimension $r$ with unitary flat normal bundle. 
Take a sufficiently fine open covering $\{U_j\}$ of $Y$. 
In this paper, we always assume that $\#\{U_j\}<\infty$ and that $U_j$ and $U_{jk}$ are simply connected and Stein for each $j$ and $k$. 
Denote by $z_j$ a coordinate of $U_j$. 
Take a sufficiently small  tubular neighborhood $V$ of $Y$ in $X$ 
and an open covering $\{V_j\}$ of $V$ with $V_j\cap Y=U_j$ for each $j$. 
By shrinking $V$ and $V_j$'s if necessary, we may assume that $U_{jk}\not=\emptyset$ iff $V_{jk}\not=\emptyset$. 

Take a defining functions system $w_j=(w_j^1, w_j^2, \dots, w_j^r)$ of $U_j$ in $V_j$. 
We regard $(z_j, w_j)$ as a coordinates system of $V_j$. 
Note that, here we denote by the same letter $z_j$ an extension of $z_j$ to $V_j$. 
In what follows, we always use the same $z_j$'s even though we often change $w_j$'s and shrink $V$ and $V_j$'s. 
More precisely, we fix a local projection $p_j\colon V_j\to U_j$ and, for any function $f$ defined on $U_j$, 
we always use the pull back $p_j^*$ for extending $f$ to $V_j$ and denote $p_j^*f$ by the same letter $f(z_j)$. 

As $N_{Y/X}$ is unitary flat, we can take a local frame $e_j=(e_j^1, e_j^2, \dotsm e_j^r)$ of the conormal bundle $N_{Y/X}^*$ on $U_j$ with $e_j=T_{jk}e_k$ for each $j$, $k$ ($T_{jk}\in U(r)$). 
By changing $w_j$ if necessary, we may assume that $dw_j=e_j$ holds on each $U_{j}$ 
(Consider a new defining functions system $M_j(z_j)\cdot w_j$ if $e_j=M_j(z_j)\cdot dw_j|_{U_j}$). 
In what follows, we always assume this condition for the system $\{w_j\}$. 
Then it follows that the expansion of the function $(\textstyle\sum_{\mu=1}^r(T_{jk})^\lambda_\mu\cdot w_k^\mu)|_{V_{jk}}$ in the variables $w_j$ is in the form of $\textstyle\sum_{\mu=1}^r(T_{jk})^\lambda_\mu\cdot w_k^\mu = w_j^\lambda+O(|w_j|^2)$, where we denote by $O(|w_j|^2)$ the higher order terms. 
Let us denote this expansion by 
\[
\sum_{\mu=1}^r(T_{jk})^\lambda_\mu\cdot w_k^\mu = w_j^\lambda+\sum_{|\alpha|\geq 2}f_{kj, \alpha}^\lambda(z_j)\cdot w_j^\alpha, 
\]
where $\alpha=(\alpha_1, \alpha_2, \dots, \alpha_r)\in(\mathbb{Z}_{\geq 0})^r$, 
$|\alpha|:=\alpha_1+\alpha_2+\cdots +\alpha_r$, 
and $w_j^\alpha:=\textstyle\prod_{\lambda=1}^r(w_j^\lambda)^{\alpha_\lambda}$. 
We also denote this expansion by 
\begin{equation}\label{eq:exp_of_Tw}
T_{jk} w_k = w_j+\sum_{|\alpha|\geq 2}f_{kj, \alpha}(z_j)\cdot w_j^\alpha, 
\end{equation}
where 
\[
w_j=
\left(
    \begin{array}{c}
      w_j^1 \\
      w_j^2 \\
	\vdots \\
	w_j^r
    \end{array}
  \right),\ 
f_{kj, \alpha}=
\left(
    \begin{array}{c}
      f_{kj, \alpha}^1 \\
      f_{kj, \alpha}^2 \\
	\vdots \\
	f_{kj, \alpha}^r
    \end{array}
  \right). 
\]

We denote by $e_j^*=(e_{j, 1}^*, e_{j, 2}^*, \dots, e_{j, r}^*)$ the dual of $e_j$ and regard it as a local frame of $N_{Y/X}$. 
For each $\alpha$ with $|\alpha|=n$, we denote by 
$e_j^\alpha$ the local section $\textstyle\prod_{\lambda=1}^r(e_j^\lambda)^{\alpha_\lambda}$ of the symmetric tensor bundle $S^nN_{Y/X}^*$. 
Then $\{e_j^\alpha\}_{|\alpha|=n}$ forms a local frame of $S^nN_{Y/X}^*$ on $U_j$. 
On each $U_{jk}$, it holds that 
$e_j^\alpha
=\textstyle\prod_{\lambda=1}^r(\textstyle\sum_{\mu=1}^r(T_{jk})^\lambda_\mu\cdot e_k^\mu)^{\alpha_\lambda}$. 
Let us denote by $\tau^\alpha_{jk, \beta}$ the coefficient of $e_k^\beta$ in the expansion of the right hand side: i.e. 
\[
e_j^\alpha
=\sum_{|\beta|=n}\tau^\alpha_{jk, \beta}\cdot e_k^\beta. 
\]

\begin{remark}\label{rmk:symtensor_flat}
Note that the matrix $(\tau^\alpha_{jk, \beta})$ need not be unitary when $r>1$, however the vector bundle $S^nN_{Y/X}^*$ itself is unitary flat. 
Here we explain the unitary flat structure of $S^nN_{Y/X}^*$ induced from the orthonormal frames $\{(U_j, e_j)\}$ of $N_{Y/X}^*$. 
Let us consider the local sections 
\[
\left.\left\{e_j^{\lambda_1}\otimes e_j^{\lambda_2}\otimes \cdots\otimes e_j^{\lambda_n}\right| \lambda_p\in\{1, 2, \dots, r\}\right\}
\]
of $\textstyle\bigotimes^nN_{Y/X}^*:=N_{Y/X}^*\otimes N_{Y/X}^*\otimes\cdots\otimes N_{Y/X}^*$ and regard it as a local frame on each $U_j$. 
Then, as the transition matrix on $U_{jk}$ is equal to $S_{jk}\otimes S_{jk}\otimes\cdots\otimes S_{jk}\in U(r^n)$, this local frame can be regarded as a orthonormal frame of the unitary flat metric induced from that of $N_{Y/X}^*$. 
By regarding each symmetric section of $\textstyle\bigotimes^nN_{Y/X}^*$ as a section of $S^nN_{Y/X}^*$ in the usual manner, we can regard $S^nN_{Y/X}^*$ as a unitary flat subbundle of $\textstyle\bigotimes^nN_{Y/X}^*$ with an orthonormal frame $\{\sqrt{n!/\alpha!}\cdot e_j^\alpha\}_\alpha$ on each $U_j$, which induces the unitary flat structure of $S^nN_{Y/X}$ ($\alpha!:=\textstyle\prod_{\lambda=1}^r\alpha_\lambda!$). 
\end{remark}


\section{The obstruction classes and the type of the pair $(Y, X)$}

\subsection{Definition of the obstruction classes}\label{section:def_of_obstr_class}

Take $\{(U_j, z_j)\}, \{(V_j, (z_j, w_j))\}, \{e_j\}$, and $\{(U_{jk}, T_{jk})\}$ as in \S \ref{section:notation}. 
In this section, we will define the obstruction class $u_n(Y, X)$ as a straightforward generalization of the Ueda class. 

\begin{definition}
We say that the system $\{(V_j, w_j)\}$ is of {\it type $n$} ($n\geq 1$) if the coefficient function $f_{kj, \alpha}$ in the expansion (\ref{eq:exp_of_Tw}) is equal to $0$ for any $\alpha$ with $|\alpha|\leq n$ on each $U_{jk}$. 
\end{definition}

Let $\{(V_j, w_j)\}$ be a system of type $n$. 
Then, by definition, the expansion (\ref{eq:exp_of_Tw}) can be written as follows: 
$T_{jk} w_k = w_j+\textstyle\sum_{|\alpha|\geq n+1}f_{kj, \alpha}(z_j)\cdot w_j^\alpha$. 
For
\[
f_{kj, n+1}=
\left(
    \begin{array}{c}
      f_{kj, n+1}^1 \\
      f_{kj, n+1}^2 \\
	\vdots \\
	f_{kj, n+1}^r
    \end{array}
  \right)
:=\sum_{|\alpha|=n+1}
\left(
    \begin{array}{c}
      f_{kj, \alpha}^1 \\
      f_{kj, \alpha}^2 \\
	\vdots \\
	f_{kj, \alpha}^r
    \end{array}
  \right)\cdot e_j^\alpha, 
\]
we can show the following: 

\begin{lemma}\label{lem:1-cocycle}
The system $\{(U_{jk}, \textstyle\sum_{\lambda=1}^re_{j, \lambda}^*\otimes f_{kj, n+1}^\lambda)\}$ satisfies the $1$-cocycle condition: i.e. $\{(U_{jk}, \textstyle\sum_{\lambda=1}^re_{j, \lambda}^*\otimes f_{kj, n+1}^\lambda)\}\in\check{Z}^1(Y, N_{Y/X}\otimes S^{n+1}N_{Y/X}^*)$. 
\end{lemma}

\begin{proof}
The lemma can be shown by summing the expansions of $T_{jk} w_k-w_j$, $T_{jk}\cdot(T_{k\ell} w_\ell-w_k)$, and $T_{j\ell}\cdot(T_{\ell j} w_j-w_\ell)$ on $V_{jk\ell}$ and comparing the terms with $w_j^\alpha$ of the both hand sides for each $\alpha$ with $|\alpha|=n+1$. 
\end{proof}

\begin{definition}
For a system $\{(V_j, w_j)\}$ of type $n$, 
we denote by $u_n(Y, X)=u_n(Y, X; \{w_j\})$ the class $[\{(U_{jk}, \textstyle\sum_{\lambda=1}^re_{j, \lambda}^*\otimes f_{kj, n+1}^\lambda)\}]\in H^1(Y, N_{Y/X}\otimes S^{n+1}N_{Y/X}^*)$ and call it the $n$-th obstruction class. 
\end{definition}

\begin{lemma}\label{lem:shougai_1}
Let $\{(V_j, w_j)\}$ be a system of type $n$ with $dw_j|_{U_j}=e_j$ for each $j$. 
Assume $u_n(Y, X; \{w_j\})=0$. 
Then there exists a system $\{(U_j, \widehat{w}_j)\}$ of type $n+1$ with $d\widehat{w}_j|_{U_j}=e_j$ for each $j$. 
\end{lemma}

\begin{proof}
From the assumption $u_n(Y, X; \{w_j\})=0$, we can take 
\[
f_{j}=
\left(
    \begin{array}{c}
      f_{j}^1 \\
      f_{j}^2 \\
	\vdots \\
	f_{j}^r
    \end{array}
  \right)
=\sum_{|\alpha|=n+1}
\left(
    \begin{array}{c}
      f_{j, \alpha}^1 \\
      f_{j, \alpha}^2 \\
	\vdots \\
	f_{j, \alpha}^r
    \end{array}
  \right)\cdot e_j^\alpha
\]
such that 
\[
f_{j, \alpha}-\sum_{|\beta|=n+1} T_{jk}f_{k, \beta}\tau^\beta_{kj, \alpha}=f_{kj, \alpha}
\]
holds on $U_{jk}$ for each $\alpha$ with $|\alpha|=n+1$. 
Define a new system $\{\widehat{w}_j\}$ by 
\[
\widehat{w}_j^\lambda:=w_j^\lambda+\sum_{|\alpha|=n+1}f_{j, \alpha}^\lambda(z_j) \cdot w_j^\alpha. 
\]
Then it follows from a simple computation that the system $\{\widehat{w}_j\}$ is of type $n+1$ with $d\widehat{w}_j|_{U_j}=dw_j|_{U_j}$, which proves the lemma. 
\end{proof}

\begin{remark}\label{rmk:line_bdl_r_setting}
Here we consider the case where $N_{Y/X}$ admits a direct decomposition $N_{Y/X}=N_1\oplus N_2\oplus\cdots\oplus N_r$ such that each $N_\lambda$ is a unitary flat line bundle on $Y$. 
It follows from Schur's lemma that such $N_\lambda$'s are unique up to ordering and isomorphism (Note that, as we mentioned in Remark \ref{rmk:injectivity_of_the_natural_map}, 
two unitary flat line bundles are isomorphic to each other iff the corresponding unitary representations coincide, see also \cite[Proposition 1 (2)]{U}). 
In this case, the transition matrix $T_{jk}$ is written in the form $T_{jk}={\rm diag}\,(t_{jk}^1, t_{jk}^2, \dots, t_{jk}^r)$ ($t_{jk}^\lambda\in U(1)$). 
Then it holds that 
\[
  \tau^\alpha_{kj, \beta} = \begin{cases}
    t_{jk}^{-\alpha}:=\prod_{\lambda=1}^r(t_{jk}^\lambda)^{-\alpha_\lambda} & (\beta=\alpha) \\
    0 & ({\rm otherwise}), 
  \end{cases}
\]
which induces a direct decomposition 
$N_{Y/X}\otimes S^{n+1}N_{Y/X}^*
=\textstyle\bigoplus_{\lambda=1}^r\textstyle\bigoplus_{|\alpha|=n+1}N_\lambda\otimes N_\alpha^{-1}$ ($N_\alpha:=\textstyle\bigotimes_{\lambda=1}^rN_\lambda^{\alpha_\lambda}$). 
Accordingly, we have a decomposition 
\[
u_n(Y, X; \{w_j\})=(u^\lambda_\alpha(Y, X; \{w_j\}))_{\lambda, \alpha}
\in \bigoplus_{\lambda=1}^r\bigoplus_{|\alpha|=n+1}H^1(Y, N_\lambda\otimes N_\alpha^{-1}) 
\]
of the $n$-th obstruction class in this case. 
It is easily observed that $u_{\alpha}^\lambda(Y, X; \{w_j\})=[\{(U_{jk}, f_{kj, \alpha}^\lambda)\}]\in H^1(Y, N_\lambda\otimes N_\alpha^{-1})$. 
\end{remark}

\subsection{Well-definedness of the obstruction classes and the type of the pair $(Y, X)$}

Take $\{U_j\}, \{V_j\}$, $\{e_j\}, \{w_j\}$, and $\{T_{jk}\}$ as in \S \ref{section:notation}. 
In this subsection, we study the dependence of the $n$-th obstruction class $u_n(Y, X; \{w_j\})$ on a system $\{(V_j, w_j)\}$ of type $n$. 

\begin{lemma}\label{lem:welldef_1}
Let $\{(V_j, w_j)\}$ and $\{(V_j, \widehat{w}_j)\}$ be systems of type $n$ such that $dw_j=d\widehat{w}_j=e_j$ holds on each $U_j$. 
Then, $u_n(Y, X; \{w_j\})=u_n(Y, X; \{\widehat{w}_j\})$. 
\end{lemma}

\begin{proof}
Let 
\begin{eqnarray}
T_{jk} w_k &=& w_j+\sum_{|\alpha|\geq n+1}f_{kj, \alpha}(z_j)\cdot w_j^\alpha, \nonumber \\
T_{jk} \widehat{w}_k &=& \widehat{w}_j+\sum_{|\alpha|\geq n+1}\widehat{f}_{kj, \alpha}(z_j)\cdot \widehat{w}_j^\alpha \nonumber
\end{eqnarray}
be the expansions as in (\ref{eq:exp_of_Tw}). 
It holds from the assumption $dw_j=d\widehat{w}_j$ that the expansion of $\widehat{w}_j$ in $w_j$ is in the form of 
$\widehat{w}_j^\lambda=w_j^\lambda+\textstyle\sum_{|\alpha|\geq 2}a_{j, \alpha}^\lambda(z_j)\cdot w_j^\alpha$, which in what follows we will denote by 
\[
\widehat{w}_j=w_j+\sum_{|\alpha|\geq 2}a_{j, \alpha}(z_j)\cdot w_j^\alpha. 
\]
Let $\nu_0$ be the maximum of the set of all $\nu\in\mathbb{Z}_{\geq 2}$ such that $a_{j, \alpha}\equiv 0$ holds for any $\alpha$ with $|\alpha|<\nu$ for each $j$. 
When $\nu_0>n+1$, it follows from $\widehat{w}_j=w_j+O(|w_j|^{n+2})$ that $f_{kj, \alpha}=\widehat{f}_{kj, \alpha}$, which proves the lemma. 
When $\nu_0=n+1$, we can calculate that 
\begin{eqnarray}\label{eq:equation_1}
T_{jk}\widehat{w}_k-\widehat{w}_j
&=&\sum_{|\beta|=n+1}\left(f_{kj, \beta}-a_{j, \beta}+T_{jk}\sum_{|\alpha|=n+1}a_{k, \alpha}\cdot \tau_{kj, \beta}^{\alpha}\right)\cdot w_j^\beta
+O(|w_j|^{n+2}). 
\end{eqnarray}
By comparing the coefficients, we obtain the equation 
\[
\left[\delta\left\{\left(U_j, \sum_{\lambda=1}^r\sum_{|\beta|=n+1}a_{j, \beta}^\lambda\cdot e_{j, \beta}^*\otimes e_j^\beta\right)\right\}\right]
=u_n(Y, X; \{w_j\})-u_n(Y, X; \{\widehat{w}_j\})
\]
in $\check{Z}^1(\{U_j\}, N_{Y/X}\otimes S^{n+1}N_{Y/X}^*)$, which proves the lemma. 
Finally, we will show the lemma for $\nu_0=\nu$ by assuming the lemma for $\nu_0=\nu+1$. 
As we may assume that $2\leq \nu\leq n$, it holds from the calculation as (\ref{eq:equation_1}) that $\{(U_{j}, \textstyle\sum_{\lambda, |\beta|=\nu}a_{j, \beta}^\lambda\cdot e_{j, \lambda}^*\otimes e_j^\beta)\}$ glue up to define a global section of $N_{Y/X}\otimes S^\nu N_{Y/X}^*$. 
Therefore, by Lemma \ref{lem:unitary_section_const} and Remark \ref{rmk:symtensor_flat}, it turns out that $a_{j, \alpha}^\lambda$ is a constant function on $U_j$ for each $\alpha$ with $|\alpha|=\nu$. 
Define a new system $\{v_j\}$ by 
$v_j^\lambda:=\widehat{w}_j^\lambda-\textstyle\sum_{|\alpha|=\nu}a_{j, \alpha}^\lambda w_j^\alpha$. 
It is easy to see that $u_n(Y, X; \{\widehat{w}_j\})=u_n(Y, X; \{v_j\})$ holds 
(Use $T_{jk}w_k=w_j+O(|w_j|^{n+1})$ and the fact that each $a_{j, \alpha}^\lambda$ is a constant). 
As $u_n(Y, X; \{w_j\})=u_n(Y, X; \{v_j\})$ holds from the lemma for $\nu_0=\nu+1$, we obtain the equation $u_n(Y, X; \{\widehat{w}_j\})=u_n(Y, X; \{w_j\})$. 
\end{proof}

\begin{proposition}\label{prop:welldef_3}
Let $\{(U_j, e_j)\}$ be a local frame of $N_{Y/X}^*$ as in \S \ref{section:notation}. 
Then one and only one of the following holds: \\
$(i)$ There exists $n\geq 1$ and a system $\{w_j\}$ of type $n$ with $dw_j|_{U_j}=e_j$ and $u_n(Y, X; \{w_j\})\not=0$. 
In this case, there is no system $\{\widehat{w}_j\}$ of type $\nu$ with $d\widehat{w}_j|_{U_j}=e_j$ for any $\nu>n$. \\
$(ii)$ For each $n\geq 1$, there exists a system $\{w_j\}$ of type $n$ with $dw_j|_{U_j}=e_j$ and $u_n(Y, X; \{w_j\})=0$. 
\end{proposition}

\begin{proof} 
Let $\{w_j\}$ be a system of type $n$ with $dw_j|_{U_j}=e_j$ for each $j$. 
Then, by Lemma \ref{lem:shougai_1} and Lemma \ref{lem:welldef_1}, 
we obtain that $u_n(Y, X; \{w_j\})=0$ iff there exists a system $\{\widehat{w}_j\}$ of type $n+1$ with $d\widehat{w}_j|_{U_j}=e_j$, which shows the proposition. 
\end{proof}

\begin{definition}
We define the type of the pair $(Y, X)$ as follows: 
${\rm type}\,(Y, X):=n$ for the case of Proposition \ref{prop:welldef_3} $(i)$, and 
${\rm type}\,(Y, X):=\infty$ for the case of Proposition \ref{prop:welldef_3} $(ii)$. 
\end{definition}

\begin{lemma}\label{prop:welldef_2}
${\rm type}\,(Y, X)$ does not depend on the choice of $\{e_j\}$. 
\end{lemma}

\begin{proof}
Let $\{(U_j, e_j)\}$ and $\{(U_j, \widehat{e}_j)\}$ be local frames of $N_{Y/X}^*$ with $e_j=T_{jk}e_k$ and $\widehat{e}_j=\widehat{T}_{jk}\widehat{e}_k$ on each $U_{jk}$ ($T_{jk}, \widehat{T}_{jk}\in U(r)$). 
Assume that there exists a system $\{w_j\}$ of type $n$ with $dw_j|_{U_j}=e_j$. 
By Proposition \ref{prop:welldef_3}, it is sufficient to show the existence of a system $\{\widehat{w}_j\}$ of type $n$ with $d\widehat{w}_j|_{U_j}=\widehat{e}_j$. 

Let 
\[
T_{jk} w_k = w_j+\sum_{|\alpha|\geq n+1}f_{kj, \alpha}(z_j)\cdot w_j^\alpha. 
\]
be the expansion (\ref{eq:exp_of_Tw}) for the system $\{w_j\}$. 
From Lemma \ref{lem:gl_flat_str_unique}, we can take $M_j\in GL_r(\mathbb{C})$ with $\widehat{e}_j=M_j\cdot e_j$. 
Note that $M_jT_{jk}=\widehat{T}_{jk}M_k$ for each $j, k$. 
Define a new system $\{\widehat{w}_j\}$ by
$\widehat{w}_j^\lambda:=\textstyle\sum_{\mu=1}^r(M_j)^\lambda_\mu\cdot w_j^\mu$. 
Then it clearly holds that $d\widehat{w}_j=\widehat{e}_j$. 
We can calculate that
\[
\widehat{T}_{jk} \widehat{w}_k
= \widehat{w}_j+\sum_{|\alpha|\geq n+1}M_j\cdot f_{kj, \alpha}\cdot \prod_{\lambda=1}^r\left(\sum_{\mu=1}^r(M_j^{-1})^\lambda_\mu \cdot \widehat{w}_j^\mu\right)^{\alpha_\nu} 
= \widehat{w}_j+O(|\widehat{w}_j|^{n+1}), 
\]
which proves the lemma. 
\end{proof}



\section{Proof of Theorem \ref{thm:main}}

\subsection{Outline}\label{section:outline_of_proof}

Let $\{U_j\}, \{V_j\}$, $\{e_j\}$, $\{T_{jk}\}$, and $\{w_j\}$ be as in \S \ref{section:notation}. 
We will prove Theorem \ref{thm:main} based on the same idea as in the proof of \cite[Theorem 3]{U} and \cite[Theorem 1.4]{KO}. 
We will construct a new system $\{u_j\}$ as the solution of a functional equation 
\begin{equation}\label{eq:func_eq}
w_j=u_j+\sum_{|\alpha|\geq 2}F_{j, \alpha}(z_j)\cdot u_j^\alpha, 
\end{equation}
where the coefficient functions 
\[
F_{j, \alpha}(z_j)=
\left(
    \begin{array}{c}
      F_{j, \alpha}^1(z_j) \\
      F_{j, \alpha}^2(z_j) \\
	\vdots \\
	F_{j, \alpha}^r(z_j)
    \end{array}
  \right)
\]
are holomorphic functions which we will construct 
in \S \ref{section:def_of_F_j_alpha} so that $\{u_j\}$ exists and satisfies $T_{jk}u_k=u_j$ on a neighborhood of $U_{jk}$ for each $j, k$ (Note that it follows from the inverse function theorem that there exists a unique solution $u_j$ if $\textstyle\sum_{|\alpha|\geq 2}F_{j, \alpha}(z_j)\cdot u_j^\alpha$ has a positive radius of convergence). 
After taking such a solution $\{u_j\}$, we obtain Theorem \ref{thm:main} $(i)$ by considering a foliation $\mathcal{F}$ whose leaves are locally defined by ``$u_j=$(constant)''. 

Theorem \ref{thm:main} $(ii)$ is also shown by considering the same functional equation (\ref{eq:func_eq}). 
Under the assumption of Theorem \ref{thm:main} $(ii)$, we will construct an initial system $\{w_j\}$ so that $\{w_j^1=0\}=V_j\cap S$ in \S \ref{section:initial_w_j_constr}. 
Starting from such an initial system, 
we will see in \S \ref{section:def_of_F_j_alpha} that one can choose coefficient functions $\{F_{j, \alpha}\}$ so that the following additional property holds for each $n\geq 2$: 
\begin{description}
\item[(Property)$_n$] $F_{j, \alpha}^1\equiv 0$ holds for any $\alpha$ with $|\alpha|=n$ and $\alpha_1=0$. 
\end{description}

Then it holds that the solution $\{u_j\}$ of the functional equation (\ref{eq:func_eq}) also satisfies $\{u_j^1=0\}=V_j\cap S$ for each $j$. 
By considering a foliation $\mathcal{G}_S$ whose leaves are locally defined by ``$u_j^1=$(constant)'', we obtain Theorem \ref{thm:main} $(ii)$. 

\begin{remark}
It may seem that the foliations $\mathcal{F}$ we will construct in the proof of Theorem \ref{thm:main} $(i)$ and $(ii)$ are different from each other at first sight. 
Actually, the solutions $u_j$'s we will obtain are different from each other. 
However, the foliation $\mathcal{F}$ itself does not depend on such differences. 
It can be shown by the following fact, which is obtained by the same arguments as in Lemma \ref{lem:welldef_1} and Proposition \ref{prop:welldef_2}: 
Let $\{u_j\}$ and $\{\widehat{u}_j\}$ be systems with $u_j=T_{jk}u_k$ and $\widehat{u}_j=\widehat{T}_{jk}\widehat{u}_k$ on each $V_{jk}$ ($T_{jk}, \widehat{T}_{jk}\in U(r)$). 
Then there exist $M_j\in GL_r(\mathbb{C})$ and $a_{j, \alpha}\in\mathbb{C}^r$ for each $j$ and $\alpha$ such that $\widehat{u}_j=M_j\cdot (u_j+\textstyle\sum_{|\alpha|\geq 2}a_{j, \alpha}\cdot u_j^\alpha)$
holds on each $V_j$. 
\end{remark}

\subsection{Construction of the initial system $\{w_j\}$}\label{section:initial_w_j_constr}

We can use any system $\{w_j\}$ with $dw_j|_{U_j}=e_j$ as an initial system for the proof of Theorem \ref{thm:main} $(i)$. 
In what follows, we explain the construction of the initial system $\{w_j\}$ under the assumption of $(ii)$. 

By Remark \ref{rmk:unitary_flat_subbundle} and the complete reducibility of the unitary representation, it follows that the short exact sequence 
$0\to N_{Y/S}\to N_{Y/X}\to N_1\to 0$ splits, where $N_1:=N_{S/X}|_Y$. 
Let $e_j^1$ be a local frame of $N_1^{-1}|_{U_j}$ with $e_j^1=t_{jk}^1e_k^1$ on $U_{jk}$ ($t_{jk}^1\in U(1)$), 
and $e_j'=(e_j^2, e_j^3, \dots, e_j^r)$ be a local frame of $N_{Y/S}^*|_{U_j}$with $e_j'=S_{jk}e_k'$ on $U_{jk}$ ($S_{jk}\in U(r-1)$). 
Take a defining function $w_j^1$ of $V_j\cap S$ in $V_j$ for each $j$. 
By a simple argument, it one can choose $\{w_j^1\}$ so that $dw_j^1=e_j^1$ holds on each $U_j$. 

\begin{lemma}\label{lem:nice_ext}
Let $w_j^1$ be a defining function of $V_j\cap S$ in $V_j$ with $dw_j^1|_{U_j}=e_j^1$ for each $j$, 
and $\{(V_j\cap S, v_j)\}=\{(V_j\cap S, (v_j^2, v_j^3, \dots, v_j^r))\}$ be a local defining functions system of $Y\subset V\cap S$ with $dv_j|_{U_j}=e_j'$ for each $j$. 
Then there exists a holomorphic function $w_j^\lambda\colon V_j\to \mathbb{C}$ with $w_j^\lambda|_{V_j\cap S}=v_j^\lambda$ for each $\lambda=2, 3, \dots, r$ such that 
$w_j:=(w_j^1, w_j^2, \dots, w_j^r)$ satisfies $dw_j=T_{jk}dw_k$ on each $U_{jk}$, where 
\[
T_{jk}:=
\left(
    \begin{array}{c|ccc}
       t_{jk}^1&0&\cdots&0\\ \hline
       0&&& \\
	 \vdots&&S_{jk}& \\
	 0&&&
    \end{array}
  \right). 
\]

\end{lemma}

\begin{proof}
Take a holomorphic function $w_j^\lambda\colon V_j\to \mathbb{C}$ with $w_j^\lambda|_{V_j\cap S}=v_j^\lambda$ for each $\lambda=2, 3, \dots, r$. 
Then the transition matrix 
$D_{jk}:=({\partial{w_j^\lambda}}/{\partial w_k^\mu}|_{U_{jk}})$ of $\{dw_j\}$ can be written in the form of 
\[
D_{jk}=
\left(
    \begin{array}{c|ccc}
       t_{jk}^1&0&\cdots&0\\ \hline
       a_{jk}^2&&& \\
	 \vdots&&S_{jk}& \\
	 a_{jk}^r&&&
    \end{array}
  \right), 
\]
where $a_{jk}^\lambda(z_j)$ is a holomorphic function defined on $U_j$. 
As 
\begin{eqnarray}
e_{j,1}^*\otimes e_j^1-e_{k,1}^*\otimes e_k^1&=&
e_{j,1}^*\otimes e_j^1-\sum_{\lambda=1}^r\sum_{\mu=1}^r(D_{kj})^1_\lambda\cdot(D_{kj}^{-1})^\mu_1\cdot e_{j,\mu}^*\otimes e_j^\lambda
=-\sum_{\mu=2}^rt_{kj}^1a_{jk}^\mu\cdot e_{j,\mu}^*\otimes e_j^1\nonumber, 
\end{eqnarray}
it holds that the extension class of the short exact sequence $0\to N_{Y/S}\to N_{Y/X}\to N_1\to 0$ is equal to $[\{(U_{jk}, -\textstyle\sum_{\mu=2}^rt_{kj}^1a_{jk}^\mu\cdot e_{j,\mu}^*\otimes e_j^1)\}]$ via the natural isomorphism ${\rm Ext}^1(N_1, N_{Y/S})\cong H^1(Y, N_1^{-1}\otimes N_{Y/S})$. 
Thus we can take $\{(U_j, (m_j^2(z_j), m_j^3(z_j), \dots, m_j^r(z_j)))\}$ such that 
\[
\left(
    \begin{array}{c}
      m_j^2 \\
      m_j^3 \\
	\vdots \\
	m_j^r
    \end{array}
  \right)
-(t_{jk}^1)^{-1}S_{jk}
\left(
    \begin{array}{c}
      m_k^2 \\
      m_k^3 \\
	\vdots \\
	m_k^r
    \end{array}
  \right)
=(t_{jk}^1)^{-1}\cdot
\left(
    \begin{array}{c}
      a_{jk}^2 \\
      a_{jk}^3 \\
	\vdots \\
	a_{jk}^r
    \end{array}
  \right)
\]
holds on each $U_{jk}$, since the short exact sequence splits. 
Let us consider 
\[
M_j:=\left(
    \begin{array}{c|ccc}
       1&0&\cdots&0\\ \hline
       -m_{j}^2&1&&0 \\
	 \vdots&&\ddots& \\
	 -m_{j}^r&0&&1
    \end{array}
  \right). 
\]
Then it holds that $M_j^{-1}T_{jk}M_k=D_{jk}$. 
Thus the lemma is shown by considering a new system $M_jw_j$. 
\end{proof}

In what follows, we use the system $\{w_j\}$ as in Lemma \ref{lem:nice_ext} and use orthonormal frame $e_j:=dw_j|_{U_j}$ whenever we consider under the assumption of Theorem \ref{thm:main} $(ii)$. 

\subsection{Preliminary observation for constructing $\{F_{j, \alpha}\}_{\alpha}$}

In this subsection, we give a heuristic explanation of how to construct $\{F_{j, \alpha}\}_{\alpha}$. 
For this purpose, we compare the expansions of the function $(T_{jk}w_k)|_{V_{jk}}$ in two manners by assuming that the solution $u_j$ of the functional equation (\ref{eq:func_eq}) exists and satisfies $T_{jk}u_k=u_j$ on $V_{jk}$. 

The first expansion is obtained by using the functional equation  (\ref{eq:func_eq}) on $V_k$ as follows: 
\begin{eqnarray}
T_{jk}w_k&=& T_{jk}u_k+\sum_{|\alpha|\geq 2}T_{jk}F_{k, \alpha}(z_k)\cdot u_k^\alpha \nonumber \\
&=& u_j+\sum_{|\alpha|\geq 2}T_{jk}F_{k, \alpha}(z_k)\cdot \sum_{|\beta|=|\alpha|}\tau_{kj, \beta}^\alpha\cdot u_j^\beta\nonumber \\
&=& u_j+\sum_{|\alpha|\geq 2}T_{jk}\left(F_{k, \alpha}(z_k(z_j, 0))+\sum_{|\gamma|\geq 1}F_{kj, \alpha, \gamma}(z_j)\cdot w_j^\gamma\right) \cdot \sum_{|\beta|=|\alpha|}\tau_{kj, \beta}^\alpha\cdot u_j^\beta\nonumber \\
&=& u_j+\sum_{|\alpha|\geq 2}\sum_{|\beta|=|\alpha|}T_{jk}F_{k, \alpha}(z_k(z_j, 0))\cdot\tau_{kj, \beta}^\alpha\cdot u_j^\beta\nonumber \\
&&+\sum_{|\alpha|\geq 2}\sum_{|\gamma|\geq 1}\sum_{|\beta|=|\alpha|}T_{jk}F_{kj, \alpha, \gamma}(z_j)\cdot w_j^\gamma\cdot \tau_{kj, \beta}^\alpha\cdot u_j^\beta\nonumber, 
\end{eqnarray}
where $F_{kj, \alpha, \gamma}$'s are the coefficients of the expansion 
\[
F_{k, \alpha}(z_k(z_j, w_j))=F_{k, \alpha}(z_k(z_j, 0))+\sum_{|\gamma|\geq 1}F_{kj, \alpha, \gamma}(z_j)\cdot w_j^\gamma.  
\]
on $V_{jk}$. 
Denoting by 
\[
h_{1, jk, \alpha}(z_j)=
\left(
    \begin{array}{c}
      h_{1, jk, \alpha}^1(z_j) \\
      h_{1, jk, \alpha}^2(z_j) \\
	\vdots \\
	h_{1, jk, \alpha}^r(z_j)
    \end{array}
  \right)
\]
the coefficient of $u_j^\alpha$ in the expansion of the function 
\[
\sum_{|\alpha|\geq 2}\sum_{|\gamma|\geq 1}\sum_{|\beta|=|\alpha|}T_{jk}F_{kj, \alpha, \gamma}\cdot \tau_{kj, \beta}^\alpha\cdot u_j^\beta\cdot \prod_{\lambda=1}^r\left(u_j^\lambda+\sum_{|\delta|\geq 2}F_{j, \delta}^\lambda\cdot u_j^\delta\right)^{\gamma_\lambda}, 
\]
we obtain 
\[
T_{jk}w_k=u_j+\sum_{|\beta|\geq 2}\left(\sum_{|\alpha|=|\beta|}T_{jk}F_{k, \alpha}(z_k(z_j, 0))\cdot\tau_{kj, \beta}^\alpha+h_{1, jk, \beta}(z_j)\right)\cdot u_j^\beta. 
\]
The second expansion is obtained by using the expansion (\ref{eq:exp_of_Tw}) as follows: 
\begin{eqnarray}
T_{jk}w_k&=& w_j+\sum_{|\alpha|\geq 2}f_{kj, \alpha}(z_j)\cdot w_j^\alpha \nonumber \\
&=& \left(u_j+\sum_{|\alpha|\geq 2}F_{j, \alpha}(z_j)\cdot u_j^\alpha\right)+\sum_{|\alpha|\geq 2}f_{kj, \alpha}(z_j)\cdot\prod_{\lambda=1}^r\left(u_j^\lambda+\sum_{|\beta|\geq 2}F_{j, \beta}^\lambda(z_j)\cdot u_j^\beta\right)^{\alpha_\lambda}\nonumber \\
&=& u_j+\sum_{|\alpha|\geq 2}\left(F_{j, \alpha}(z_j)+h_{2, jk, \alpha}(z_j)\right)\cdot u_j^\alpha, \nonumber 
\end{eqnarray}
where we are denoting by 
\[
h_{2, jk, \alpha}(z_j)=
\left(
    \begin{array}{c}
      h_{2, jk, \alpha}^1(z_j) \\
      h_{2, jk, \alpha}^2(z_j) \\
	\vdots \\
	h_{2, jk, \alpha}^r(z_j)
    \end{array}
  \right)
\]
the coefficient of $u_j^\alpha$ in the expansion of the function 
\[
\sum_{2\leq |\gamma|<n}f_{kj, \gamma}\prod_{\lambda=1}^r\left(u_j^\lambda+\sum_{2\leq |\beta|<n}F_{j, \beta}^\lambda\cdot u_j^\beta\right)^{\gamma_\lambda}. 
\]
By comparing these two expansions, it is observed that the coefficient functions $\{F_{j, \beta}\}$ should be chosen so that the equation 
\[
F_{j, \beta}(z_j)-\sum_{|\alpha|=|\beta|}T_{jk}F_{k, \alpha}(z_k)\cdot\tau_{kj, \beta}^\alpha=h_{1, jk, \beta}(z_j)-h_{2, jk, \beta}(z_j)
\]
holds on $U_{jk}$ for each $\beta$. 
Note that this equation means that 
\[
\delta\left\{\left(U_j, \sum_{\lambda, |\beta|=n}F_{j, \beta}^\lambda\cdot e_{j, \lambda}\*\otimes e_j^\beta\right)\right\}
=\left\{\left(U_{jk}, \sum_{\lambda, |\alpha|=n} \left(h^\lambda_{1, jk, \alpha}-h^\lambda_{2, jk, \alpha}\right)\cdot e_{j, \lambda}^*\otimes e_j^\alpha\right)\right\} 
\]
holds in $\check{Z}^1(\{U_j\}, N_{Y/X}\otimes S^nN_{Y/X}^*)$ for each $n\geq 2$. 
Accordingly, we need $[\{(U_{jk}, H_{jk, n})\}]=0\in H^1(Y, N_{Y/X}\otimes S^nN_{Y/X}^*)$ for each $n\geq 2$, where 
\[
H_{jk, n}:=\sum_{\lambda=1}^r\sum_{|\alpha|=n}\left(h_{1, jk, \alpha}-h_{2, jk, \alpha}\right)\cdot e_{j, \lambda}^*\otimes e_j^\alpha, 
\]
which we will actually show in the next subsection by using the assumption ${\rm type}\,(Y, X)=\infty$. 

\subsection{Inductive construction of $F_{j, \alpha}$}\label{section:def_of_F_j_alpha}

Based on the observation in the previous subsection, we construct the coefficient functions $F_{j, \alpha}$. 
In the following inductive construction, the following properties of  $H_{jk, n}$ are essential: 
$H_{jk, 2}=-\textstyle\sum_{\lambda=1}^re_{j, \lambda}^*\otimes f_{kj, 2}^\lambda$ holds and $H_{jk, n}$ depends only on $\{F_{j, \alpha}\}_{|\alpha|<n}$ for each $n\geq 3$. 
These properties are easily shown by the definition of $H_{jk, n}$. 

\subsubsection*{Step $1$ (The construction of $\{F_{j, \alpha}\}_{|\alpha|=2}$)}

For each $\beta$ with $|\beta|=2$, we take $\{(U_j, F_{j, \beta})\}$ as a solution of the equation 
\[
F_{j, \beta}(z_j)-\sum_{|\alpha|=2}T_{jk}F_{k, \alpha}(z_k)\cdot\tau_{kj, \beta}^\alpha=-f_{kj, \beta}. 
\]
Note that there actually exists a solution of this equation, since $u_1(Y, X)=0$ holds. 
Strictly speaking, we choose appropriate solution $\{(U_j, F_{j, \alpha})\}$ of the above equation by using \cite[Lemma 3]{U} (=\cite[Lemma 2]{KS}) or \cite[Lemma 4]{U} as we will explain the details in \S \ref{section:norm_estim}. 

\begin{claim}\label{claim:step1}
Fix $\{F_{j, \beta}\}_{|\beta|=2}$ as above. Then the following holds: \\
$(a)$ For any choice of the remaining coefficient functions $\{F_{j, \alpha}\}_{|\alpha|>2}$, the solution $\{u_j\}$ of the functional equation (\ref{eq:func_eq}) is a system of type $2$ if exists. \\
$(b)$ Under the assumption of Theorem \ref{thm:main} $(ii)$, we can take $\{F_{j, \beta}\}_{|\beta|=2}$ with {\bf{\rm (Property)}$_{2}$}. 
\end{claim}

\begin{proof}
$(a)$ is shown by comparing the expansions of $(T_{jk}w_k)|_{V_{jk}}$ in two manners as in the previous section. 
We skip the details here since the computation is almost the same as (and much easier than) that in the proof of Lemma \ref{lem:H} below. 

Under the assumption of Theorem \ref{thm:main} $(ii)$, $f_{kj, \alpha}^1\equiv 0$ holds for each $\alpha$ with $\alpha_1=0$, since $t_{jk}^1w_k^1$ is divisible by $w_j^1$ (recall that we are using an initial system $\{w_j\}$ as in \S \ref{section:initial_w_j_constr} in this setting). 
Thus, by considering the decomposition $S^mN_{Y/X}^*=\textstyle\bigoplus_{\ell=0}^m\left(N_1^{-\ell}\otimes S^{m-\ell}N_{Y/S}^*\right)$, the defining equation of $\{F_{j, \alpha}^1\}_{|\alpha|=2, \alpha_1=0}$ can be rewritten by the equation 
\[
\delta\left(\left\{\left(U_j, \sum_{|\alpha|=2, \alpha_1=0}F_{j, \alpha}^1\cdot e_{j, 1}^*\otimes e_j^\alpha\right)\right\}\right)=0
\]
in $\check{Z}^1(\{U_j\}, N_1\otimes S^2N_{Y/S}^*)$, which proves the assertion $(b)$. 
\end{proof}

\subsubsection*{Step $(n-1)$ (The construction of $\{F_{j, \alpha}\}_{|\alpha|=n}$)}

After choosing $F_{j, \alpha}$ for each $\alpha$ with $|\alpha|<n$, we take $\{(U_j, F_{j, \beta})\}_{|\beta|=n}$ as a solution of the equation 
\[
F_{j, \beta}(z_j)-\sum_{|\alpha|=n}T_{jk}F_{k, \alpha}(z_k)\cdot\tau_{kj, \beta}^\alpha=h_{1, jk, \beta}-h_{2, jk, \beta}. 
\]
Here we use the fact that $h_{1, jk, \beta}$ and $h_{2, jk, \beta}$ depend only on $\{F_{j, \alpha}\}_{|\alpha|<n}$. 
The existence of a solution $\{(U_j, F_{j, \alpha})\}$ is assured by Lemma \ref{lem:H} $(a)$ below. 
Strictly speaking, we choose appropriate $\{(U_j, F_{j, \alpha})\}$ from the solutions by using \cite[Lemma 3]{U} (=\cite[Lemma 2]{KS}) or \cite[Lemma 4]{U} as we will explain the details in \S \ref{section:norm_estim}. 

\begin{lemma}\label{lem:H}
Let $\{F_{j, \alpha}\}_{|\alpha|\leq n-1}$ be as in Step $(n-2)$. Then the following holds: \\
$(a)$ $[\{(U_{jk}, H_{jk, n})\}]=0\in H^1(Y, N_{Y/X}\otimes S^nN_{Y/X}^*)$. \\
$(b)$ Let $\{F_{j, \alpha}\}_{|\alpha|=n}$ be as above. Then, for any choice of the remaining coefficient functions $\{F_{j, \alpha}\}_{|\alpha|>n}$, the solution $\{u_j\}$ of the functional equation (\ref{eq:func_eq}) is of type $n$ if exists. \\
$(c)$ Under the assumption of Theorem \ref{thm:main} $(ii)$, we can take $\{F_{j, \beta}\}_{|\beta|=n}$ with {\bf{\rm (Property)}$_{n}$}. 

\end{lemma}

\begin{proof}
Fix $\{F_{j, \alpha}\}_{|\alpha|\geq n}$ and consider the solution $\{u_j\}$ of the functional equation (\ref{eq:func_eq}). 
By Lemma \ref{lem:H} $(b)$ for Step $(n-2)$, it turns out that $\{u_j\}$ is of type $n-1$ and thus 
\[
u_k^\alpha=\prod_{\lambda=1}^r\left(\sum_{\mu=1}^r(T_{kj})^\lambda_\mu u_j^\mu+O(|u_j|^n)\right)^{\alpha_\lambda}
=\sum_{|\beta|=|\alpha|}\tau_{kj, \beta}^\alpha\cdot u_j^\beta +O(|u_j|^{n+1})
\]
holds for each $\alpha$ with $|\alpha|\geq 2$. 
Consider the expansions 
\begin{eqnarray}
T_{jk}w_k&=& T_{jk}u_k+\sum_{|\alpha|\geq 2}T_{jk}F_{k, \alpha}(z_k)\cdot u_k^\alpha \nonumber \\
&=& T_{jk}u_k+\sum_{|\alpha|\geq 2}T_{jk}F_{k, \alpha}(z_k)\cdot \sum_{|\beta|=|\alpha|}\tau_{kj, \beta}^\alpha\cdot u_j^\beta+O(|u_j|^{n+1})\nonumber \\
&=& T_{jk}u_k+\sum_{2\leq|\beta|\leq n}\left(\sum_{|\alpha|=|\beta|}T_{jk}F_{k, \alpha}(z_k(z_j, 0))\cdot \tau_{kj, \beta}^\alpha+h_{1, jk, \beta}(z_j)\right)u_j^\beta+O(|u_j|^{n+1})\nonumber 
\end{eqnarray}
and 
\begin{eqnarray}
T_{jk}w_k
&=& u_j+\sum_{2\leq |\alpha|\leq n}\left(F_{j, \alpha}(z_j)+h_{2, jk, \alpha}(z_j)\right)\cdot u_j^\alpha+O(|u_j|^{n+1}). \nonumber 
\end{eqnarray}
By comparing these, we obtain 
\begin{eqnarray}\label{eq:exp_uchidome}
&&T_{jk}u_k-u_j \\
&=&\sum_{|\beta|=n}\left(F_{j, \beta}(z_j)-\sum_{|\alpha|=n}T_{jk}F_{k, \alpha}(z_k(z_j, 0))\cdot\tau_{kj, \beta}^\alpha-h_{1, jk, \beta}(z_j)+h_{2, jk, \beta}(z_j)\right)\cdot u_j^\beta\nonumber \\
&&+O(|u_j|^{n+1}). \nonumber
\end{eqnarray}
By considering this equation (\ref{eq:exp_uchidome}) in the case where $F_{j, \alpha}\equiv 0$ for each $\alpha$ with $|\alpha|\geq n$, we obtain $u_{n-1}(Y, X; \{u_j\})=[\{(U_{jk}, -H_{jk, n})\}]$. Thus the assertion $(a)$ follows from the assumption $u_{n-1}(Y, X)=0$. 
The assertion $(b)$ also follows directly from the equation (\ref{eq:exp_uchidome}). 

Under the assumption of Theorem \ref{thm:main} $(ii)$, it follows from the argument as in the proof of Claim \ref{claim:step1} that the defining equation of $\{F_{j, \alpha}^1\}_{|\alpha|=n, \alpha_1=0}$ can be rewritten by the equation 
\[
\delta\left(\left\{\left(U_j, \sum_{|\alpha|=n, \alpha_1=0}F_{j, \alpha}^1\cdot e_{j, 1}^*\otimes e_j^\alpha\right)\right\}\right)=\left\{\left(U_{jk}, \sum_{|\alpha|=n, \alpha_1=0} (h_{1, jk, \alpha}^1-h_{2, jk, \alpha}^1)\cdot e_{j, 1}^*\otimes e_j^\alpha\right)\right\}
\]
in $\check{Z}^1(\{U_j\}, N_1\otimes S^nN_{Y/S}^*)$. 
Thus it is sufficient for proving the assertion $(c)$ to show $h_{1, jk, \alpha}^1\equiv 0$ and $h_{2, jk, \alpha}^1\equiv 0$ for each $\alpha$ with $|\alpha|=n$ and $\alpha_1=0$, 
which can be easily checked from Lemma \ref{lem:H} $(c)$ in Step $(\nu)$ for each $\nu\leq n-2$. 
\end{proof}

\subsection{Norm estimate for $F_{j, \alpha}$ in a special setting}\label{section:norm_estim}

In this subsection, we estimate the norm of $F_{j, \alpha}$ in order to show the convergence of the functional equation (\ref{eq:func_eq}). 
Here we treat a special case where $(Y, X)$ is as in Remark \ref{rmk:line_bdl_r_setting}: i.e. $N_{Y/X}$  admits a direct decomposition $N_{Y/X}=N_1\oplus N_2\oplus\cdots\oplus N_r$. 
Note that, in this case, the defining equation of each $\{F_{j, \alpha}\}$ is rewritten by the equation 
\begin{equation}\label{eq:eq_for_def_F_j_alpha_in_special_cases}
\delta(\{(U_j, F_{j, \alpha}^\lambda)\})=\{(U_{jk}, h_{1, jk, \alpha}^\lambda-h_{2, jk, \alpha}^\lambda)\}
\end{equation}
in $\check{Z}^1(Y, N_\lambda\otimes N_\alpha^{-\alpha})$. 
We additionally assume that $N_{Y/X}$ is the holomorphically trivial vector bundle $\mathbb{I}^{(r)}_Y$ of rank $r$ or $N_{Y/X}\in \mathcal{S}^{(r)}(Y):=\textstyle\bigcup_{A>0}\mathcal{S}^{(r)}_A(Y)$. 

\subsubsection{The case where $N_{Y/X}\cong \mathbb{I}^{(r)}_Y$}\label{section:normestim_triv_case}

Here we consider the case where $N_{Y/X}\cong \mathbb{I}^{(r)}_Y$. 
Note that, in this case, $N_\lambda\otimes N_\alpha^{-\alpha}\cong\mathbb{I}^{(1)}_Y$ holds for any $\lambda$ and $\alpha$. 
Fix $U_j^*\Subset U_j$ with $\textstyle\bigcup_jU_j^*= Y$. 
Take a constant $K:=K(\mathbb{I}^{(1)}_Y)$ as in \cite[Lemma 3]{U} (=\cite[Lemma 2]{KS}): 
i.e. for any $1$-cocycle $a=\{(U_{jk}, a_{jk})\}\in \check{Z}^1(\{U_j\}, \mathcal{O}_Y)$ with $\|a\|:=\textstyle\max_{j, k}\textstyle\sup_{U_{jk}}|a_{jk}|<\infty$ which is cohomologous to zero, there exists a $0$-cochain $b=\{(U_j, b_j)\}\in \check{C}^0(\{U_j\}, \mathcal{O}_Y)$ such that $a$ is the coboundary of $b$ and that $\|b\|:=\textstyle\max_{j}\textstyle\sup_{U_{j}}|b_j|\leq K\|a\|$. 
Take also a positive number $M$ larger than $\textstyle\max_j\textstyle\max_{\lambda}\textstyle\sup_{V_j}|w_j^\lambda|$ and $\textstyle\max_{jk}\textstyle\max_{\lambda}\textstyle\sup_{V_{jk}}|w_{k}^\lambda|$, and a sufficiently large positive number $R$ so that $\{(z_j, w_j)\mid z_j\in U_j\cap U_k^*, |w_j|<R^{-1}\}\subset V_k$ holds for each $j, k$. 
By using these constants, let us consider the formal series $A(X)=A(X^1, X^2, \dots, X^r)=\textstyle\sum_{|\alpha|\geq 2}A_\alpha X^\alpha$ defined by 
\begin{eqnarray}\label{eq:def_of_A}
&&A(X) \\
&=&2KA(X)\cdot\left(-1+\prod_{\lambda=1}^r\frac{1}{1-R(X^\lambda+A(X))}\right) \nonumber\\
&&+2KM\cdot\left(-1-R(X^1+X^2+\cdots+X^r+rA(X))+\prod_{\lambda=1}^r\frac{1}{1-R(X^\lambda+A(X))}\right). \nonumber
\end{eqnarray}
Note that, by the inductive argument on $|\alpha|$, it is easy to see that each coefficient $A_\alpha$ is determined uniquely and is a positive real number. 

\begin{lemma}
The formal series $A(X)$ has a positive radius of convergence. 
\end{lemma}

\begin{proof}
Let us consider 
\begin{eqnarray}
P(X, Y)&:=&-Q(X, Y)\cdot Y+2KY\cdot\left(-Q(X, Y)+1\right)\nonumber \\
&&+2KM\cdot\left(Q(X, Y)\cdot\left(-1-R(X^1+X^2+\cdots+X^r+rY)\right)+1\right),  \nonumber 
\end{eqnarray}
where $Q(X, Y):=\textstyle\prod_{\lambda=1}^r(1-R(X^\lambda+Y))\in\mathbb{C}[X, Y]$. 
As $P(0, Y)=-Y+O(Y^2)$, we can apply the implicit function theorem to obtain a holomorphic function $a(X)$ defined on a neighborhood of the origin of $\mathbb{C}^r$ with $a(0)=0$ and $P(X, a(X))\equiv 0$. 
This means that $a(X)$ satisfies the equation (\ref{eq:def_of_A}) and thus we obtain $a(X)=A(X)$ on a neighborhood of the origin, which proves the lemma. 
\end{proof}

In what follows, we will show that one can choose the coefficient functions $\{F_{j, \alpha}\}$ as in the previous section so that $A(X)$ is a dominating series of the function equation (\ref{eq:func_eq}). 

First, we will show the existence of the solution $\{(U_j, F_{j, \alpha})\}_{|\alpha|=2}$ of the equation (\ref{eq:eq_for_def_F_j_alpha_in_special_cases}) with $\textstyle\max_\lambda\|\{(U_j, F_{j, \alpha}^\lambda)\}\|\leq A_\alpha$ for each $\lambda=1, 2, \dots, r$ and $\alpha$ with $|\alpha|=2$ 
($\|\{(U_j, F_{j, \alpha}^\lambda)\}\|:=\textstyle\max_j\textstyle\sup_{U_j}|F_{j, \alpha}^\lambda|$). 
Note that $A_\alpha=2KMR^2$ holds for each $\alpha$ with $|\alpha|=2$. 
By Cauchy estimate, we obtain the inequality $\textstyle\max_j\textstyle\sup_{U_j\cap U_k^*}|f_{kj, \alpha}^\lambda|\leq MR^2$. Combining this inequality and the argument as in \cite[p. 599]{U}, we obtain $\|\{(U_{jk}, f_{kj, \alpha}^\lambda)\}\|\leq 2MR^2$ 
($\|\{(U_j, f_{kj, \alpha}^\lambda)\}\|:=\textstyle\max_{j, k}\textstyle\sup_{U_{jk}}|f_{kj, \alpha}^\lambda|$). 
Thus we obtain $\|\{(U_j, F_{j, \alpha}^\lambda)\}\|\leq A_\alpha$ from the definition of the constant $K$. 

Next, we will show the existence of the solution $\{(U_j, F_{j, \alpha})\}_{|\alpha|=n}$ of the equation (\ref{eq:eq_for_def_F_j_alpha_in_special_cases}) with $\textstyle\max_\lambda\|\{(U_j, F_{j, \alpha}^\lambda)\}\|\leq A_\alpha$ for each $\lambda=1, 2, \dots, r$ and $\alpha$ with $|\alpha|=n$ by assuming the assertion for $|\alpha|<n$. 
Take $\alpha$ with $|\alpha|=n$. 
Then, it follows from the inductive assumption that $\textstyle\max_{j, k}\textstyle\sup_{U_j\cap U_k^*}|h_{1, jk, \alpha}^\lambda|$ is bounded by the coefficient of $X^\alpha$ in the expansion of 
\[
\sum_{2\leq |\gamma|< n}\sum_{|\beta|\geq 1}\max_{j, k}\sup_{U_j\cap U_k^*}\left|F_{kj, \gamma, \beta}^\lambda\right|\cdot X^\gamma\cdot 
\prod_{\lambda=1}^r\left(X^\lambda+A(X)\right)^{\beta_\lambda}. 
\]
By Cauchy estimate, it is bounded by the coefficient of $X^\alpha$ in the expansion of 
\[
\sum_{2\leq |\gamma|< n}\sum_{|\beta|\geq 1}A_\gamma R^{|\beta|}\cdot X^\gamma\cdot 
\prod_{\lambda=1}^r\left(X^\lambda+A(X)\right)^{\beta_\lambda}
=A(X)\cdot\left(-1+\prod_{\lambda=1}^r\frac{1}{1-R(X^\lambda+A(X))}\right). 
\]
From a similar argument, it can be seen that $\textstyle\max_{j, k}\textstyle\sup_{U_j\cap U_k^*}|h_{2, jk, \alpha}^\lambda|$ is bounded by the coefficient of $X^\alpha$ in the expansion of 
\[
M\cdot\left(-1-R(X^1+X^2+\cdots+X^r+rA(X))+\prod_{\lambda=1}^r\frac{1}{1-R(X^\lambda+A(X))}\right). 
\]
Thus, by the argument as in \cite[p. 599]{U} and the defining function (\ref{eq:def_of_A}) of $A(X)$, we obtain the inequality $\|\{(U_{jk}, h_{1, jk, \alpha}^\lambda-h_{2, jk, \alpha}^\lambda)\}\|\leq K^{-1}A_\alpha$ for each $\lambda$. 
Therefore the assertion follows from the definition of the constant $K$. 

\subsubsection{The case where $N_{Y/X}\in\mathcal{S}^{(r)}(Y)$}\label{section:normestim_E_1_case}

When $N_{Y/X}\in\mathcal{S}^{(r)}(Y)$, by using \cite[Lemma 4]{U} instead of \cite[Lemma 3]{U}, the same arguments as in \S \ref{section:normestim_triv_case} can be carried out after replacing 
the defining function (\ref{eq:def_of_A}) of $A(X)$ with 
\begin{eqnarray}
&&\sum_{|\alpha|\geq 2}\varepsilon_{|\alpha|-1}^{-1}A_\alpha X^\alpha\nonumber \\
&=&2A(X)\cdot\left(-1+\prod_{\lambda=1}^r\frac{1}{1-R(X^\lambda+A(X))}\right)\nonumber \\
&&+2M\cdot\left(-1-R(X^1+X^2+\cdots+X^r+rA(X))+\prod_{\lambda=1}^r\frac{1}{1-R(X^\lambda+A(X))}\right),  \nonumber
\end{eqnarray}
where 
\[
\varepsilon_n^{-1}:=\frac{1}{K}\min_{\alpha\in\mathbb{Z}^r,\ |\alpha|=n}d(\mathbb{I}_Y^{(1)}, N_\alpha) 
\]
($K$ is the constant as in \cite[Lemma 4]{U}, see also \cite[\S 4.6]{U} for the details). 
Thus, for proving the convergence of the functional equation (\ref{eq:func_eq}), it is sufficient to see the convergence of the formal series $A(X)$ with the above new defining equation. 

Consider $B(Y):=Y+A(Y, Y, \dots, Y)=Y+\textstyle\sum_{n=2}^\infty B_nY^n$, where $B_n=\textstyle\sum_{|\alpha|=n}A_\alpha$ for each $n\geq 2$. 
As it can be easily seen that $A_\alpha\geq 0$ holds for each $\alpha$, we have $A_\alpha\leq B_{|\alpha|}$. 
Therefore, for showing the convergence of $A(X)$, it is sufficient to show that $B(Y)$ has a positive radius of convergence. 
By considering $X^1=X^2=\dots=X^r=Y$, we obtain the defining function of $B(Y)$ as follows: 
\begin{eqnarray}
&&\sum_{n=2}^\infty\varepsilon_{n-1}^{-1}B_n Y^n\nonumber \\
&=&2(B(Y)-Y)\cdot\left(-1+\frac{1}{(1-RB(Y))^r}\right)+2M\cdot\left(-1-rRB(Y)+\frac{1}{(1-RB(Y))^r}\right).  \nonumber
\end{eqnarray}
Also consider another formal series $\widehat{B}(Y)=Y+\textstyle\sum_{n=2}^\infty \widehat{B}_nY^n$ defined by 
\begin{eqnarray}
&&\sum_{n=2}^\infty\varepsilon_{n-1}^{-1}\widehat{B}_n Y^n\nonumber \\
&=&2\widehat{B}(Y)\cdot\left(-1+\frac{1}{(1-R\widehat{B}(Y))^r}\right)+2M\cdot\left(-1-rR\widehat{B}(Y)+\frac{1}{(1-R\widehat{B}(Y))^r}\right). \nonumber
\end{eqnarray}
As it clearly holds that $\widehat{B}_n\geq B_n$ for each $n\geq 2$, it is sufficient to show that $\widehat{B}(Y)$ has a positive radius of convergence. 
According to Siegel's argument (\cite{S}, see also \cite[Lemma 5]{U}), it is sufficient to see the following two properties of $\{\varepsilon_n\}$: 
$(a)$ There exists a positive number $A$ such that $\varepsilon_n<(2n)^A$ for any $n\geq 1$, and 
$(b)$ $\varepsilon^{-1}_{n-m}\leq \varepsilon^{-1}_n+\varepsilon^{-1}_m$ for any $n>m$. 
The property $(a)$ directly follows from the assumption that $N_{Y/X}\in \mathcal{S}^{(r)}(Y)$. 
The property $(b)$ can be shown by
\[
\varepsilon^{-1}_n+\varepsilon^{-1}_m=\frac{1}{K}\left(d(\mathbb{I}_Y^{(1)}, N_{\alpha^{(n)}})+d(\mathbb{I}_Y^{(1)}, N_{\alpha^{(m)}})\right)
\geq\frac{1}{K}d(N_{\alpha^{(n)}}, N_{\alpha^{(m)}})=\frac{1}{K}d(\mathbb{I}_Y^{(1)}, N_{\alpha^{(n)}-\alpha^{(m)}}), 
\]
where $\alpha^{(n)}\in \mathbb{Z}^r$ is an element which attains the minimum in the definition of $\varepsilon^{-1}_n$. 

\subsection{End of the proof for the special setting}\label{section:end_of_the_proof_for_sp}
Assume that $N_{Y/X}$ is holomorphically trivial or $N_{Y/X}\in \mathcal{S}^{(r)}(Y)$ holds. 
Then, by choosing the coefficient functions $\{F_{j, \alpha}\}$ as above, we can deduce from the inverse function theorem that there exists a solution $\{u_j\}$ of the functional equation (\ref{eq:func_eq}). 
By shrinking $V$ if necessary, we may assume that $u_j$ is defined on $V_j$ for each $j$. 
From Lemma \ref{lem:H} $(b)$, it holds that the solutions satisfy $u_j=T_{jk}u_k$ on each $V_{jk}$. 
Thus the theorem for this special case follows from the arguments as we already explained in \S \ref{section:outline_of_proof}. 

\subsection{Proof for the general setting}

When $N_{Y/X}\in\mathcal{E}_0^{(r)}(Y)$, consider $G:={\rm ker}\,\rho$, where $\rho=\rho_{N_{Y/X}}$ is the unitary representation corresponding to $N_{Y/X}$.  
Fix a tubular neighborhood $V$ of $Y$ in $X$ and regard $G$ as a normal subgroup of $\pi_1(V, *)$ by the natural isomorphism $\pi_1(Y, *)\cong\pi_1(V, *)$. 
From the assumption, there exists a finite normal covering $\pi\colon\widetilde{V}\to V$ corresponding to $G\subset \pi_1(V, *)$. 
Denote by  $\widetilde{Y}$ the preimage $\pi^{-1}(Y)$. 
Then it is clear from the construction that 
$N_{\widetilde{Y}/\widetilde{V}}=(\pi|_{\widetilde{Y}})^*N_{Y/X}$ is holomorphically trivial. 
Let us denote $\pi^{-1}(V_j)$ by $\widetilde{V}_j$ and $\widetilde{V}_j\cap \widetilde{Y}$ by $\widetilde{U}_j$. 
We may assume that $\widetilde{U}_j$ is the union of $d$ copies of $U_j$, where $d$ is the degree of the map $\pi$. 
Consider the local defining functions system $\{\widetilde{w}_j\}$ defined by $\widetilde{w}_j:=(\pi|_{\widetilde{V}_j})^*w_j$. 

By Lemma \ref{lem:well_def_sub} below, $(\widetilde{Y}, \widetilde{V})$ is of infinite type. 
Thus, from the result we showed in \S \ref{section:end_of_the_proof_for_sp}, 
we can solve the functional equation (\ref{eq:func_eq}) with initial system $\{\widetilde{w}_j\}$ on each $\widetilde{V}_j$ to obtain a local defining functions system $\{\widetilde{u}_j\}$ of $\widetilde{Y}$ in $\widetilde{V}$ with $\widetilde{u}_j=T_{jk}\widetilde{u}_k$ on each $\widetilde{V}_{jk}$. 
Note that, as $\widetilde{w}_j=\widetilde{u}_j+O(|\widetilde{u}_j|^2)$, it holds that $d\widetilde{u}_j|_{\widetilde{U}_j}=(\pi|_{\widetilde{U}_j})^*e_j$. 

Define a function $u_j$ on $V_j$ by 
\[
(\pi|_{\widetilde{V}_j})^*u_j=\frac{1}{d}\sum_{\nu=1}^di_\nu^*\widetilde{u}_j, 
\]
where $\{i_1, i_2, \dots, i_d\}$ is the set of deck transformations of $\pi$. 
Clearly it holds that $du_j|_{U_j}=e_j$ and $\{u_j= 0\}=U_j$ hold, which means that $\{u_j\}$ is a local defining functions system of $Y$. 
It is also easy to see that $u_j=T_{jk}u_k$ holds on each $V_{jk}$, which shows the assertion $(i)$ for the case where $N_{Y/X}\in\mathcal{E}^{(r)}_0(Y)$. 
Under the assumption in Theorem \ref{thm:main} $(ii)$, 
it follows from {\bf (Property)}$_n$'s that we may assume that $\widetilde{u}_j^1$ is a defining function of $\widetilde{V}_j\cap \pi^{-1}(S)$ in $\widetilde{V}_j$ with $d\widetilde{u}_j^1=\pi^*((1+O(|v_j|))\cdot dw_j^1)$ on each $\widetilde{V}_j\cap\pi^{-1}(S)$. 
Therefore $u_j^1$ is a defining function of $V_j\cap S$ by shrinking $V$ if necessary, which proves the assertion $(ii)$ for the case where $N_{Y/X}\in\mathcal{E}^{(r)}_0(Y)$. 

When $N_{Y/X}\in\mathcal{E}_1^{(r)}(Y)$, there exists a finite normal covering $\pi\colon\widetilde{Y}\to Y$ such that $(\pi|_{\widetilde{Y}})^*N_{Y/X}\in \mathcal{S}^{(r)}(\widetilde{Y})$. 
The theorem for this case is shown by the same argument as above by using this map $\pi$. 

\begin{lemma}\label{lem:well_def_sub}
Let $\pi, \widetilde{Y}, \widetilde{V}$ be as above. Then  
${\rm type}\,(Y, X)={\rm type}\,(\widetilde{Y}, \widetilde{V})$ holds. 
\end{lemma}

\begin{proof}
Let 
\[
T_{jk} w_k = w_j+\sum_{|\alpha|\geq n+1}f_{kj, \alpha}\cdot w_j^\alpha. 
\]
be the expansion (\ref{eq:exp_of_Tw}) for the system $\{w_j\}$. Then, by pulling it back by $\pi$, we obtain 
\[
T_{jk} \widetilde{w}_k = \widetilde{w}_j+\sum_{|\alpha|\geq n+1}(\pi|_{\widetilde{U}_j})^*f_{kj, \alpha}\cdot \widetilde{w}_j^\alpha 
\]
on each $\widetilde{V}_{jk}$. 
Thus, $\{\widetilde{w}_j\}$ is a system of type $n$ and 
\[
u_n(\widetilde{Y}, \widetilde{V}; \{\widetilde{w}_j\})=
(\pi|_{\widetilde{Y}})^*u_n(Y, X; \{w_j\})
\]
holds. 
Therefore we obtain the lemma, since the map $(\pi|_{\widetilde{Y}})^*\colon H^1(Y, N_{Y/X}\otimes S^nN_{Y/X}^*)\to H^1(Y, N_{\widetilde{Y}/\widetilde{V}}\otimes S^nN_{\widetilde{Y}/\widetilde{V}}^*)$ is injective. 
\end{proof}



\section{Proof of Corollary {\ref{cor:main}}}

As $D_\lambda$'s intersect to each other transversally along $C$, it follows that 
$N_{C/X}=\textstyle\bigoplus_{\lambda=1}^{n-1}N_\lambda$ and  $N_{C/D_1}=\textstyle\bigoplus_{\lambda=2}^{n-1}N_\lambda$, where $N_\lambda:=N_{D_\lambda/X}|_Y$ for each $\lambda=1, 2, \dots, n-1$. 
Note that each $N_\lambda$ is isomorphic to $L|_C$ and thus it is an element of $\mathcal{E}^{(1)}_1(C)$. 
Thus, as $N_{C/X}\cong \oplus^rL|_C\in \mathcal{E}^{(r)}_1(C)$ and 
$H^1(C, N_\lambda\otimes N_\alpha^{-1})\cong H^1(C, L|_C^{-|\alpha|+1})=0$ 
hold, it follows that $(X, S:=D_1, Y:=C)$ satisfies the assumption of Theorem \ref{thm:main} $(ii)$. 
Therefore we obtain that the line bundle $L=[D_1]$ admits a unitary flat metric on a neighborhood $V$ of $C$. 
By considering the (regularized) minimum of this unitary flat metric on $L|_V$ and the Bergman type singular Hermitian metric $(\textstyle\sum_{\lambda=1}^{n-1}h_\lambda^{-1})^{-1}$, we can construct a $C^\infty$ Hermitian metric on $L$ with semi-positive curvature 
($h_\lambda$ is a singular Hermitian metric on $L$ with $|f_\lambda|_{h_\lambda}^2\equiv 1$, where $f_\lambda\in H^0(X, L)$ is a section with ${\rm div}(f_\lambda)=D_\lambda$, see also \cite[Corollary 3.4]{K2} for the regularized minimum construction). 



\section{Examples}

\subsection{Deformation spaces of projective manifolds}

Let $B$ be a domain of $\mathbb{C}^r$ which includes the origin and $\pi\colon X\to B$ be a deformation of projective manifolds: i.e. $X$ is a holomorphic manifold of dimension $n+r$ and $\pi$ is a proper holomorphic surjective submersion whose fiber $\pi^{-1}(x)$ is a projective manifold of dimension $n$ for each $x\in X$. 
Denote by $Y$ the central fiber $\pi^{-1}(0)$. 
Let us assume that $Y$ is a smooth fiber for simplicity. 
In this case, $N_{Y/X}$ is holomorphically trivial. 

Take a coordinate $x=(x^1, x^2, \dots, x^r)$ of $\mathbb{C}^r$. 
Then, by considering a global defining functions system $w:=\{w^\lambda:=\pi^*x^\lambda\}$ of $Y$, 
it is easily seen that the pair $(Y, X)$ is of infinite type. 
In this case, Theorem \ref{thm:main} $(i)$ is easily checked. 
Indeed, the foliation $\mathcal{F}$ in this case is the one which is induced by the fibration $\pi$. 
In what follows, we give a simple proof of Theorem \ref{thm:main} $(ii)$ for this fundamental example. 

Let $S\subset X$ be a non-singular hypersurface such that $Y\subset S$ and $N_{Y/S}$ is unitary flat. Let us consider the line bundle $[S]$. 
It holds that $[S]|_Y=[S]|_S|_Y=N_{S/X}|_Y$. 
As we have already mentioned in \S \ref{section:initial_w_j_constr}, it follows from Remark \ref{rmk:unitary_flat_subbundle} and the complete reducibility of the unitary representations that $N_{S/X}|_Y$ is unitary flat line bundle. 
Therefore, $[S]|_Y$ is unitary flat and thus it is topologically trivial. 
As the first Chern class $c_1([S]|_{\pi^{-1}(x)})$ depends continuously on $x$, it holds that $[S]|_L$ is also topologically trivial for each leaf $L$ of $\mathcal{F}$. 
Assume that $L\not\subset S$. 
Then, as the divisor $S|_L$ is an effective divisor on a projective manifold such that the corresponding line bundle is topologically trivial, it follows that $S\cap L=\emptyset$. 
Therefore we obtain that $S=\pi^{-1}(\overline{S})$ holds, where $\overline{S}:=\pi(S)$. 
By shrinking $B$ and choosing appropriate $x$, 
we may assume that $\overline{S}=\{x=(x^1, x^2, \dots, x^r)\in B\mid x^1=0\}$. 
Then we can construct the foliation $\mathcal{G}_S$ as in Theorem \ref{thm:main} $(ii)$ by considering ``$w^1=$(constant)". 

\subsection{Projective bundles}

Let $M$ be a compact complex manifold and $E$ be a holomorphic vector bundle on $M$ of rank $r+1$. 
Assume that there exists a subbundle $F\subset E$ of rank $r$ such that $F$ is a unitary flat vector bundle and the quotient bundle $L:=E/F$ is the holomorphically trivial line bundle. 
In this subsection, we consider the projective bundle $X:=\mathbb{P}(E)$ and the section $Y\subset X$ of $\pi\colon X\to M$ defined by the natural map $p\colon E\to L$. 

Fix an open covering $\{U_j\}$ of $M$ and take a local frame $(e_j^1, e_j^2, \dots, e_j^r)$ of $F$ with $e_j^\lambda=\textstyle\sum_{\mu=1}^r(S_{jk}^{-1})^\lambda_\mu\cdot e_k^\mu$, or equivalently, $e_{j, \lambda}^*=\textstyle\sum_{\mu=1}^r(S_{jk})_\lambda^\mu\cdot e_{k, \mu}^*$ on each $U_{jk}$ ($S_{jk}\in U(r)$). 
By extending these appropriately, we obtain a local frame $e_j=(e_j^0, e_j^1, \dots, e_j^r)$ of $E$ with $e_j^*=\widetilde{S}_{jk}e_k^*$ on each $U_{jk}$, where 
\[
\widetilde{S}_{jk}:=\left(
    \begin{array}{c|ccc}
       1&a_{jk, 1}&\cdots&a_{jk, r}\\ \hline
       0&&& \\
	 \vdots&&S_{jk}& \\
	 0&&&
    \end{array}
  \right). 
\]
Here the function $a_{jk, \lambda}$ is a holomorphic function defined on $U_{jk}$ for each $\lambda$. 
Fix a neighborhood $V_j$ of $\pi^{-1}(U_j)\cap Y$ in $\pi^{-1}(U_j)$ and a coordinate $z_j$ of $U_j$. 
For each $w_j=(w_j^1, w_j^2, \dots, w_j^r)\in\mathbb{C}^r$, consider the map 
\[
(z_j, w_j)\mapsto \left[e_{j, 0}^*(z_j)+\sum_{\lambda=1}^rw_j^\lambda\cdot e_{j, \lambda}^*(z_j)\right]
\]
and regard $(z_j, w_j)$ as a coordinates system of $V_j$ by this map. 
Then we obtain 
\[
w_k^\mu=\frac{\sum_{\lambda=1}^r(S_{jk})_\lambda^\mu\cdot w_j^\lambda}{1+\sum_{\lambda=1}^ra_{jk, \lambda} w_j^\lambda}
=\sum_{\lambda=1}^r(S_{jk})_\lambda^\mu\cdot w_j^\lambda-\sum_{\lambda=1}^r\sum_{\nu=1}^r(S_{jk})_\lambda^\mu\cdot a_{jk, \nu}w_j^\lambda w_j^\nu+O(|w_j|^3), 
\]
and thus 
\begin{eqnarray}
\sum_{\mu=1}^r(S_{jk}^{-1})^p_\mu \cdot w_k^\mu
&=&w_j^p-
\sum_{\nu=1}^r a_{jk, \nu}w_j^p w_j^\nu+O(|w_j|^3)\nonumber
\end{eqnarray}
on each $V_{jk}$, 
which can be regarded as the expansion (\ref{eq:exp_of_Tw}) for the local defining functions system $\{w_j\}$ with the transition matrix $T_{jk}:=S_{jk}^{-1}$ of $N_{Y/X}$ (Note that $N_{Y/X}^*\cong F$). 

Set $\varepsilon_j:=dw_j$. Then it follows from the above expansion that the first obstruction class $u_1(Y, X; \{w_j\})$ is defined by 
\[
-\sum_{p=1}^r\sum_{\nu=1}^r a_{jk, \nu}\cdot \varepsilon_{j, p}^*\otimes\varepsilon_j^p\cdot \varepsilon_j^\nu. 
\]
On the other hand, it follows from the arguments as in Lemma \ref{lem:nice_ext} that 
$-\textstyle\sum_{\nu=1}^r a_{jk, \nu}\cdot \varepsilon_j^\nu$ can be regarded as the extension class $\delta(1)\in H^1(Y, N_{Y/X}^*) (\cong H^1(M, L^{-1}\otimes F))$ of the short exact sequence $0\to F\to E\to L\to 0$. 
Thus we can conclude that $u_1(Y, X)$ is the image of ${\rm id}_{N_{Y/X}}\otimes\delta(1)\in H^1(Y, {\rm End}(N_{Y/X})\otimes N_{Y/X}^*)\cong H^1(Y, N_{Y/X}\otimes N_{Y/X}^*\otimes N_{Y/X}^*)$ by the map 
\[
H^1(Y, N_{Y/X}\otimes N_{Y/X}^*\otimes N_{Y/X}^*)\to H^1(Y, N_{Y/X}\otimes S^2N_{Y/X}^*) 
\]
induced from the natural map $N_{Y/X}^*\otimes N_{Y/X}^*\to S^2N_{Y/X}^*$. 
In what follows, we regard each section of $S^2N_{Y/X}^*$ as a symmetric section of $N_{Y/X}^*\otimes N_{Y/X}^*$ and $S^2N_{Y/X}^*$ as a unitary flat subbundle of $N_{Y/X}^*\otimes N_{Y/X}^*$. 
Then the natural map $N_{Y/X}^*\otimes N_{Y/X}^*\to S^2N_{Y/X}^*$ can be regarded as the map $\varepsilon_j^\mu\otimes \varepsilon_j^\nu\mapsto {\rm Sym}(\varepsilon_j^\mu\otimes \varepsilon_j^\nu):=\textstyle\frac{1}{2}\left(\varepsilon_j^\mu\otimes \varepsilon_j^\nu+\varepsilon_j^\nu\otimes \varepsilon_j^\mu\right)$. 
By using this, it clearly holds that $u_1(Y, X)=0$ iff $s_*(\delta(1))=0\in H^1(Y, N_{Y/X}\otimes N_{Y/X}^*\otimes N_{Y/X}^*)$ holds, where $s_*$ is the map induced from 
\[
s\colon N_{Y/X}^*\to N_{Y/X}\otimes N_{Y/X}^*\otimes N_{Y/X}^*:\ 
\varepsilon_j^\nu\mapsto 
\frac{1}{2}\sum_{\mu=1}^r\varepsilon_{j, \mu}^*\otimes \left(\varepsilon_j^\mu\otimes \varepsilon_j^\nu+\varepsilon_j^\nu\otimes \varepsilon_j^\mu\right). 
\]
As we can regard $N_{Y/X}^*$ as a unitary flat subbundle (and thus a direct component by the complete reducibility of the unitary representation) of $N_{Y/X}\otimes N_{Y/X}^*\otimes N_{Y/X}^*$ via $s$, we can conclude that 
$u_1(Y, X)=0$ holds iff $\delta(1)=0$ holds, or equivalently, $0\to F\to E\to L\to 0$ splits. 

Therefore, it holds that the pair $(Y, X)$ is of type $1$ if $0\to F\to E\to L\to 0$ does not split. 
When $0\to F\to E\to L\to 0$ splits, it can be easily seen that $T_{jk}w_k=w_j$ holds on each $V_{jk}$, which shows that the pair $(Y, X)$ is of infinite type in this case. 

\subsection{The blow up of a del Pezzo manifold at a general point (Proof of Corollary \ref{cor:dPmfd_metric})}\label{section:dPmfd}

Let  $(V, L)$ be a del Pezzo manifold of degree $1$: 
i.e. $V$ is a projective manifold of dimension $n$ and $L$ is an ample line bundle on $V$ with $K_V^{-1}\cong L^{n-1}$ and the self-intersection number $(L^n)$ is equal to $1$. 
From \cite[6.4]{F}, it holds that ${\rm dim}\,H^0(V, L)=n$. 
Take general elements $D_1, D_2, \dots, D_n\in |L|$. 
By \cite[4.2]{F} and $(D_1, D_2, \dots, D_n)=(L^n)=1$, it holds that 
the intersection $\textstyle\bigcap_{\lambda=1}^nD_\lambda$ is a point, which we denote by $p$. It is clear that $D_\lambda$'s intersect each other transversally at $p$. 
From this fact and Bertini's theorem, we may assume that each $D_\lambda$ is non-singular. 

Consider an sequence of the subvarieties $V_n:=V, V_{n-1}:=D_1, V_{n-2}:=D_1\cap D_2, \cdots, V_1:=D_1\cap D_2\cdots, \cap D_{n-1}$. 
Denote by $L_\lambda$ the restriction $L|_{V_\lambda}$ for each $\lambda=1, 2, \cdots, n-1$. 
Note that it follows from a simple inductive argument that $(V_\lambda, L_\lambda)$ is also a del Pezzo manifold of degree $1$ for each $\lambda$. 
Especially, for $\lambda=1$, it holds that $V_1$ is an elliptic curve and ${\rm deg}\,L_1=1$. 
Take $q\in V_1$ and denote by $\pi\colon X\to V$ the blow-up at $q$. 
Let us denote by $E$ the exceptional divisor, by $\widetilde{D}_\lambda$ the strict transform $(\pi^{-1})_*D_\lambda$, and by $Y$ the strict transform $(\pi^{-1})_*V_1$. 
Then it is clear that 
$\widetilde{D}_1, \widetilde{D}_2, \dots, \widetilde{D}_{n-1}\in |\widetilde{L}|$, where $\widetilde{L}:=\pi^*L\otimes\mathcal{O}_X(-E)$, and that $\widetilde{D}_\lambda$'s intersect each other transversally along $Y$. 
Thus we can apply Corollary \ref{cor:main} to this example to obtain Corollary \ref{cor:dPmfd_metric}.

\subsection{An example of an infinite type pair which does not admit $\mathcal{F}$ as in Theorem \ref{thm:main}}

In \cite[\S 5.4]{U}, Ueda constructed a pair $(C, S)$ of a surface $S$ and a compact curve $C$ of genus $g\geq 1$ embedded in $S$ with unitary flat normal bundle such that $(C, S)$ is infinite type, 
however there does not exist a foliation $\mathcal{F}$ as in Theorem \ref{thm:main}. 
Here we will construct such a pair for the case where the codimension $r$ is greater than $1$. 

Let $(C, S)$ be as above with $g=1$. 
By shrinking $S$ to a tubular neighborhood of $C$ if necessary, we may assume $\pi_1(C, *)\cong \pi_1(S, *)$. 
Denote by $\rho:=\rho_{N_{C/S}}$ the unitary representation of $\pi_1(C, *)$ corresponding to the unitary flat line bundle $N_{C/S}$ and 
by $\widetilde{S}\to S$ the universal covering of $S$. 
Set $X:=\widetilde{S}\times \mathbb{C}^{r-1}/\sim$, where $\sim$ is the relation defined by 
\[
(z, (v_1, v_2, \dots, v_{r-1}))\sim (\gamma z, (\rho(\gamma)v_1, \rho(\gamma)v_2, \dots, \rho(\gamma)v_{r-1}))
\]
for each $(z, (v_1, v_2, \dots, v_{r-1}))\in \widetilde{S}\times\mathbb{C}^{r-1}$ and $\gamma\in\pi_1(S, *)$. 
We denote by $Z$ the submanifold $\widetilde{S}\times \{(0, 0, \dots, 0)\}/\sim$ of $X$ and 
by $Y$ the submanifold $\widetilde{C}\times \{(0, 0, \dots, 0)\}/\sim$ of $Z$, 
where $\widetilde{C}\subset\widetilde{S}$ is the universal covering of $C$. 
Note that $(Y, Z)$ is naturally isomorphic to $(C, S)$. 
As $N_{Y/X}\cong N_{C/S}^{\oplus r}$ holds and $N_{C/S}$ is non-torsion, $H^1(Y, N_{Y/X}\otimes S^nN_{Y/X}^*)=0$ holds for each $n\geq 2$. 
Therefore we obtain that $(Y, X)$ is of infinite type. 

Assume that there exists a local defining functions system $\{(V_j, w_j)\}$ of $Y$ with $w_j=T_{jk}w_k$ on each $V_{jk}$ ($T_{jk}\in U(r)$), where $\{V_j\}$ is as in \S \ref{section:notation}. 
Take a local frame $\{e_j\}$ of $N_{Y/X}^*$ such that $e_j=t_{jk}e_k$ on each $U_{jk} (=Y\cap V_{jk})$, 
where $t_{jk}\in U(1)$ is a transition function for some local frames of $N_{C/S}^*$. 
Let $A_j\colon U_j\to GL_r(\mathbb{C})$ be a holomorphic function defined by $e_j=A_j\cdot dw_j|_{U_j}$. 
By considering $\{(U_j, A_j)\}$ as a global section of the vector bundle ${\rm End}\,(N_{Y/X}^*)$, it follows from Lemma \ref{lem:unitary_section_const} that each $A_j$ is a constant function (see also \cite[\S 1 Proposition 1]{Se}). 
Thus, by replacing $w_j$ with $A_j\cdot w_j$, we may assume that $T_{jk}={\rm diag}\, (t_{jk}, t_{jk}, \dots, t_{jk})$. 
For a fixed index $j_0$, it is clear that $w_{j_0}^\lambda|_{V_{j_0}\cap Z}\not\equiv 0$ for some $\lambda$. 
Without loss of generality, we may assume $\lambda=1$. 
Set $f_j:=w_j^1|_{V_j\cap Z}$ for each $j$. 
Then, as $f_j=t_{jk}f_k$ on each $V_{jk}\cap Z$ and $f_{j_0}\not\equiv 0$, 
we obtain that $f_j\not\equiv 0$ for any $j$ and that the divisors ${\rm div}(f_j)$ glue up to define a divisor $D$ of $V\cap Z$ ($V=\textstyle\cup_jV_j$). 
Let $D=aY+\textstyle\sum_{\nu=1}^\ell b_\nu W_\nu$ be the irreducible decomposition of $D$ ($a, b_\nu>0$, note that we may assume $\ell<\infty$ by shrinking $V$ if necessary). 
As the line bundle $[D]$ is unitary flat, the intersection number $(D, Y)$ can be computed as follows: $(D, Y)={\rm deg}\,[D]|_Y=0$. 
The self-intersection number $(Y, Y)$ is also equal to $0$, since $(Y, Y)={\rm deg}\, N_{Y/Z}={\rm deg}\, N_{C/S}$. 
Therefore it holds that $(W_\nu, Y)=0$ for each $\nu$, which means that we may assume that $D=aY$ by shrinking $V$ if necessary. 
Thus it holds that the system $\{(V_j\cap Z, f_j)\}$ induces a foliation $\mathcal{F}$ on a neighborhood of $C$ in $S$ as in Theorem \ref{thm:main}, which contradicts to the property of the pair $(C, S)$. 



\section{discussion}

In this section, we list some remaining problems. 

In \cite[\S 3]{U}, the neighborhood structure of $Y$ is investigated also for the pair $(Y, X)$ of finite type for the case where $r=1$. 
According to \cite[Theorem 1]{U}, $Y$ admits a fundamental system of strongly pseudoconcave
neighborhoods. 
As an analogy, it seems to be natural to ask the following question for example. 

\begin{question}
Let $X$ be a compact complex manifold and $Y\subset X$ be a compact complex submanifold with unitary flat normal bundle such that the pair $(Y, X)$ is of type $n<\infty$. 
What kind of psh functions do there exist on $X\setminus Y$? 
\end{question}

One of the most interesting application of \cite[Theorem 1]{U} is the classification of the pairs $(Y, X)$ of finite type such that $X$ is a projective surface and $Y$ is an elliptic curve \cite[\S 6]{N}. We are also interested in a higher dimensional analogy of this result: 

\begin{question}
Classify the pairs $(Y, X)$ of finite type such that $X$ is a projective manifold and $Y$ is an elliptic curve with unitary flat normal bundle. 
\end{question}

We are also interested in some concrete examples. 
In our context, the example of the blow-up of $\mathbb{P}^2$ at nine points is one of the most interesting examples, see \cite{A}, \cite{B}, and \cite[\S 1]{D}. 
The example we treated in \S \ref{section:dPmfd} is a natural generalization of this example. 
From this point of view, it seems to be natural to ask the following: 

\begin{question}[higher dimensional analogue of {\cite[Question 1.2]{K3}}]
Let $(V, L), C=V_1, q, Y, X$ be as in Corollary \ref{cor:dPmfd_metric} and \S \ref{section:dPmfd}. 
Is there a point $q\in C$ such that $K_X^{-1}$ admits no $C^\infty$ Hermitian metric with semi-positive curvature, 
or that $Y$ does not admit a pseudoflat neighborhood system? 
\end{question}

In \cite{K}, we studied the neighborhood structure of a submanifold $Y$ of $X$ with codimension $r=2$. 
Under the assumption of the existence of a hypersurface $S$ of $X$ with unitary flat normal bundle which includes $Y$ as a submanifold, we posed the obstruction classes $u_{n, m}(Y, S, X)\in H^1(N_{S/X}|_Y^{-n}\otimes N_{Y/S}^{-m})$. 
Thereafter, we found a mistake in the proof of \cite[Theorem 1]{K}, 
which is corrected as \cite[Theorem 1.4]{KO} by using a new obstruction classes 
$v_{n, m}(Y, S, X)\in H^1(N_{S/X}|_Y^{-n}\otimes N_{Y/S}^{-m+1})$. 

\begin{question}
What is the relation between $(u_{n, m}(Y, S, X), v_{n, m}(Y, S, X))$ and $u_n(Y, X)$ we defined in \S \ref{section:def_of_obstr_class} ? 
\end{question}




\end{document}